\documentclass{article}
\usepackage[letterpaper,margin=1.0in]{geometry}
\usepackage{amsthm}

\usepackage{amsmath,amsfonts,bm,amssymb}

\def\eqref#1{equation~\ref{#1}}

\def\1{\bm{1}}

\def\vzero{{\bm{0}}}

\def\vd{{\bm{d}}}

\def\vh{{\bm{h}}}

\def\vu{{\bm{u}}}
\def\vv{{\bm{v}}}
\def\vw{{\bm{w}}}
\def\vx{{\bm{x}}}
\def\vy{{\bm{y}}}
\def\vz{{\bm{z}}}

\def\mA{{\bm{A}}}
\def\mB{{\bm{B}}}

\def\mI{{\bm{I}}}

\def\mO{{\bm{O}}}

\DeclareMathAlphabet{\mathsfit}{\encodingdefault}{\sfdefault}{m}{sl}
\SetMathAlphabet{\mathsfit}{bold}{\encodingdefault}{\sfdefault}{bx}{n}

\def\gO{{\mathcal{O}}}

\def\sO{{\mathbb{O}}}

\def\sR{{\mathbb{R}}}

\newtheorem{thm}{Theorem}[section]
\newtheorem{dfn}{Definition}[section]

\newtheorem{lem}{Lemma}[section]
\newtheorem{asm}{Assumption}[section]

\newtheorem{cond}{Condition}[section]

\usepackage{mdframed}
\usepackage{hyperref}
\usepackage{enumitem}
\usepackage[utf8]{inputenc} 
\usepackage[T1]{fontenc}    
\usepackage{hyperref}       
\usepackage{url}            
\usepackage{booktabs}       
\usepackage{amsfonts}       
\usepackage{nicefrac}       
\usepackage{microtype}      
\usepackage{graphicx}       
\usepackage[numbers,sort]{natbib}
\usepackage{multirow}
\usepackage{bbding}
\usepackage{caption}
\graphicspath{{media/}}     
\usepackage{color}
\usepackage{algorithmic,algorithm}
\usepackage{nicefrac}
\usepackage{cancel}
\usepackage{graphicx} 
\usepackage{booktabs}
\usepackage{makecell}
\usepackage{multirow}
\usepackage{amssymb,amsthm}
\usepackage{tcolorbox}
\usepackage[table]{xcolor}
\usepackage{pifont}
\usepackage{multirow, hhline}
\definecolor{mydarkgreen}{RGB}{39,130,67}

\definecolor{mydarkred}{RGB}{192,47,25}

\usepackage{thmtools}
\usepackage{thm-restate}

\title{Near-Optimal Convex Optimization with \\
Lazy Second-Order Oracles}

\author{Xinliang Zhang \textsuperscript{* 1} \qquad Lesi Chen \textsuperscript{* 1}\\
Chengchang Liu \textsuperscript{2}
\qquad Jingzhao Zhang \textsuperscript{1} \\
\vspace{-2mm} \\
\normalsize{\textsuperscript{1}IIIS, Tsinghua University\quad 
\textsuperscript{2} Department of Artificial Intelligence, Westlake University. 
} \\
\normalsize{ \texttt{\{xinliang23,chenlc23\}@mails.tsinghua.edu.cn,  }} \\
\normalsize{\texttt{7liuchengchang@gmail.com,
jingzhaoz@mail.tsinghua.edu.cn
}}}

\begin{document}
\maketitle
\begingroup
\begin{NoHyper}
\renewcommand\thefootnote{*}
\footnotetext{Equal contributions.}
\end{NoHyper}
\endgroup
\begin{abstract}
This paper studies the complexity of convex optimization using lazy second-order oracles (Doikov, Chayti, and Jaggi, ICML 2023), where an algorithm queries gradients every iteration and Hessians once per $m$ iterations. Under this setting, we show a lower bound of $\Omega(m+ m^{1/7} \epsilon^{-2/7})$ on the number of total iterations to find an $\epsilon$-solution using a novel block zero-chain construction. Then we propose a novel method that achieves a new upper bound of $\tilde \gO(m+ m^{1/7} \epsilon^{-2/7})$, which significantly improves the prior one (Chen, Liu, Luo, and Zhang, COLT 2026) of $\tilde \gO(m+ m^{13/21} \epsilon^{-2/7})$ and is tight up to logarithmic factors.
\end{abstract}

\section{Introduction}
In this work, we consider the complexity of solving the convex optimization problem
\begin{align} \label{prob:main}
\min_{\vx \in \sR^d} f(\vx).
\end{align}
Following the seminal work by \citet{nemirovskij1983problem}, we consider the oracle complexity for convex optimization. In this framework, we assume that every algorithm is interacting with an oracle $\sO(\vx)$ that returns the local information of query $\vx$. Two fundamental oracles are: (1) first-order oracle $\sO_1(\vx) = \nabla f(\vx)$ returns the gradient information; (2) the second-order oracle $\sO_2(\vx) = (\nabla f(\vx), \nabla^2 f(\vx))$ additionally returns a Hessian matrix that reflects the curvature information. The core question in oracle complexity is the minimal number of oracle queries $T$ required to find an $\epsilon$-solution to Problem (\ref{prob:main}).

Both the optimal complexities are known in the \textit{standard oracle model} that counts the complexity by
\begin{align} \label{eq:oracle-model}
    \# \sO_1(\vx)~~~~~\text{or}~~~~~\# \sO_2(\vx)
\end{align}
for first- or second-order methods, respectively.
In 1983, \citet{nesterov1983method} discovered the optimal first-order methods, which can find an $\epsilon$-solution to the problem in the first-order complexity of $\gO(\epsilon^{-1/2})$ that matches the lower bound in \citet{nemirovskij1983problem} up to constants. Over the past two decades, a series of works studying how to use second-order information for further acceleration. In 2006, \citet{nesterov2006cubic} introduced the cubic regularized Newton (CRN) method as the first globally convergent second-order method, based on which \citet{nesterov2008accelerating} in 2008 proposed the accelerated CRN with the second-order oracle complexity of $\gO(\epsilon^{-1/3})$.  In 2013, \citet{monteiro2013accelerated} proposed a faster accelerated Newton
proximal extragradient (A-NPE) method with the complexity of 
$\gO(\epsilon^{-2/7} \log \epsilon^{-1})$. In 2022, \citet{carmon2022optimal,kovalev2022first} improved A-NPE to achieved the optimal complexity of $\gO(\epsilon^{-2/7})$ that matches the lower bound in \citet{arjevani2019oracle}.

Although the optimal complexities have been established for the standard oracle model,
we argue that it may not reflect the computational costs of an algorithm.
Firstly, the standard oracle model restricts the algorithm to access either the first-order oracle or the second-order oracle. 
While in practice, the algorithm may access $\sO_2(\vx)$  at some iterations and only $\sO_1(\vx)$ for other iterations.
Secondly, the standard oracle model does not distinguish the costs between the Hessian and the gradient. 
The computational cost of a Hessian matrix is usually much higher the one of a gradient. 
Therefore, a more reasonable setup is to assume that the cost of $\sO_2(\vx)$ is $m$ times the cost of $\sO_1(\vx)$ and to use the following \textit{cost model}, which measures the complexity of an algorithm by: 
\begin{align} \label{eq:cost-model}
    \# \sO_1(\vx)+ m \cdot\# \sO_2(\vx).
\end{align} 
Under this new cost model, the prior optimal first- and second-order methods have the complexity of $\gO(\epsilon^{-1/2})$ and $\gO(m \epsilon^{-2/7})$, respectively. 
We observe that the complexity of this cost model can be reduced to count the complexity of \textit{$m$-lazy second-order oracle} which returns $\sO_2(\vx)$ once every $m$ iterations and $\sO_1(\vx)$ for the other iterations~\citep{doikov2023second,chayti2023unified}.
Very recently, \citet{chen2026faster} gives an $\tilde \gO(m +m^{13/21} \epsilon^{-2/7})$ upper bound of $m$-lazy second-order oracle for convex optimization, 
which strictly improves all  existing second-order methods and achieves an optimal dependency in $\epsilon$. 
However, 
the optimal dependency on $m$  remains unknown.

\paragraph{Our contributions.} In this paper, we fully resolve the optimal complexity for convex optimization using lazy second-order oracles, with the contributions summarized as follows.
\begin{enumerate}
    \item We introduce a block zero-chain that results in a lower bound of $\Omega(m+m^{1/7} \epsilon^{-2/7} )$ for any $m$-lazy second-order algorithms. In contrast, prior applicable lower bounds based on standard zero-chains \citep{nesterov2018lectures,arjevani2019oracle} can not distinguish the information of first- and second-order oracles, and can only give the trivial lower bound of $\Omega(\epsilon^{-2/7})$ even for $m \ge $1.
    \item We also propose a novel second-order method that achieves the upper bound of $\tilde \gO( m+ m^{1/7} \epsilon^{-2/7} )$, which matches our lower bound up to logarithmic factors. Our result builds upon the previous framework by \citet{chen2026faster} that reduces the problem to solving a series of cubic proximal subproblems, while we design a faster sub-solver based on a preconditioned accelerated gradient descent with segment search on an extrapolated objective.
\end{enumerate}

\paragraph{AI Usage.} The authors derived the lower bound of $\Omega(m+m^{1/7} \epsilon^{-2/7})$ and utilized this benchmark to evaluate AI reasoning capabilities. While ChatGPT 5.5 Pro and 5.6 Pro Sol both failed to yield a non-trivial dependence on $m$ for the lower bound, ChatGPT 6 Astra ultimately recovered the exact same construction as the authors. Meanwhile, the upper bound was established through multiple iterations with these AI models. ChatGPT 5.6 Pro Sol first established an upper bound of $\tilde{\mathcal{O}}(m + m^{3/7} \epsilon^{-2/7})$ by improving the cubic proximal sub-solver via accelerated inexact second-order methods \citep{chen2026optimal}. 
Later, through extensive interactions with the newly released ChatGPT 6 Astra, the upper bound was first sharpened to $\tilde{\mathcal{O}}(m + m^{1/3} \epsilon^{-2/7})$, and finally to the near-optimal complexity of $\tilde{\mathcal{O}}(m + m^{1/7} \epsilon^{-2/7})$. The initial responses by AI are publicly available~\footnote{https://chatgpt.com/share/6aae90a9-3250-83e8-82cb-2e532e8c0ab1}. 



 


\section{Preliminaries}

\paragraph{Notations.} We use the notation $\|\cdot\|$ to denote both the Euclidean norm of a vector and the spectral norm of a matrix. Given a matrix $\mB\succeq \mO_d$, we use
$\|\mathbf{v}\|_\mB^2=\mathbf{v}^{\top}\mB\mathbf{v}$ to denote the matrix $\mB$-norm. Logarithms are natural unless their base is specified. We adopt the standard big $\gO(\,\cdot\,)$ and $\Omega(\,\cdot\,)$ notations, and use $\tilde \gO(\,\cdot\,)$ or $\tilde \Theta(\,\cdot\,)$ to hide the logarithmic factors.

\paragraph{Problem Setup.}
In this paper, we consider second-order convex optimization under the following assumptions. First, the convexity of the function is described as follows.
\begin{asm} \label{asm:convex}
We assume that $f$ is convex, \textit{i.e.},
\[
f(\vx') \ge f(\vx) + \langle \nabla f(\vx), \vx' - \vx \rangle , \quad \forall \vx, \vx' \in \sR^d.
\]
\end{asm}

Second, as a standard assumption for second-order methods \citep[Chapter 4]{nesterov2018lectures}, we assume that the objective function in Problem (\ref{prob:main}) has Lipschitz continuous Hessians.
\begin{asm} \label{asm:Hessian-Lip}
    We assume that $\nabla^2 f(\vx)$ exists and is $L$-Lipschitz continuous ($L>0$), \textit{i.e.},
    \begin{align*}
        \Vert \nabla^2 f(\vx) - \nabla^2 f(\vx') \Vert \le L \Vert \vx-  \vx' \Vert, \quad \forall \vx,\vx' \in \sR^d.
    \end{align*}
\end{asm}

Third, as typically required by accelerated algorithms \citep{nesterov1983method,nesterov2008accelerating,monteiro2013accelerated}, we assume that there exists at least one solution $\vx^*$ for the problem.
\begin{asm} \label{asm:bounded-minimizer}
We assume $f$ has at least one solution $\vx^* \in \arg \min_{\vx \in \sR^d} f(\vx)$.
\end{asm}

Under the above assumptions, we study the complexity of an algorithm to approximately solve Problem~(\ref{prob:main}). Formally, the complexity of an algorithm is defined as the minimal number of iterations to find the following notions of $\epsilon$-solution.

\begin{dfn} \label{dfn:sub-opt-point}
We call $\hat \vx \in \sR^d$  an $\epsilon$-solution to Problem (\ref{prob:main}) if $f(\hat \vx) - f^* \le \epsilon$.
\end{dfn}


\section{Lower Complexity Bound}

As a starting point, it is not hard to observe that, under the cost model (\ref{eq:cost-model}), the number of gradient queries of any optimal method must be $m$ times that of the Hessians. Equivalently,
\begin{align} \label{eq:constraint-ho-fo}
 \# \sO_1(\vx) = m \cdot \# \sO_2(\vx)   
\end{align}
Otherwise, if the algorithm uses fewer gradient evaluations such that $\# \sO_1(\vx) < m \cdot \# \sO_2(\vx)$, then it is strictly dominated by completing the missing gradient queries, which have access to more information without incurring additional costs up to constants.
Symmetrically, if it uses fewer Hessian evaluations such that $\# \sO_2(\vx) <\# \sO_1(\vx) /m$, then it is also inferior to one that fills the remaining Hessian queries.

Therefore, we only need to consider second-order algorithms under constraint (\ref{eq:constraint-ho-fo}), for which we provide lower bounds in the following: We first give lower bounds for the second-order algorithms~\citep{doikov2023second,chayti2023unified} in Section~\ref{subsec:lb-lazy-so}, which will also imply the same lower bounds for general second-order algorithms under constraint~(\ref{eq:constraint-ho-fo}) via a simulation-based reduction in Section~\ref{subsec:lb-general}.

\subsection{Lower Bound for Lazy Second-Order Algorithms} \label{subsec:lb-lazy-so}


We first give lower bounds for the lazy second-order algorithms studied by~\citet{doikov2023second,chayti2023unified}, where the algorithm only queries the Hessian information periodically once every $m$ iterations. Such algorithms are visualized in Figure \ref{fig:lazy-hess}.

\begin{figure}[htbp] 
	\centering 
	\setlength{\arrayrulewidth}{0.5mm}
	\renewcommand{\arraystretch}{1.8}
	\setlength{\tabcolsep}{3pt}
	\arrayrulecolor[HTML]{66BB6A}
	\begin{tabular}{| c | c | c | c | c |}
		\hhline{-----}
		\cellcolor[HTML]{FFFDE7}
		$\boldsymbol{\nabla^2 f(\vx_0)}$ & 
		\multicolumn{3}{c|}{  
			\cellcolor[HTML]{FFFDE7}
			\large
			$\overset{\textbf{reuse Hessian}}{\boldsymbol{\longrightarrow}}$	 
		}  &
		\cellcolor[HTML]{FFFDE7}
		$\!\boldsymbol{\nabla^2 f(\vx_m)}\!$ \\
		\hhline{-----}
		\cellcolor[HTML]{FFFDE7}
		$\boldsymbol{ \nabla f(\vx_0) }$ &
		\cellcolor[HTML]{FFFDE7}
		$\boldsymbol{ \nabla f(\vx_1)} $ & 
		\cellcolor[HTML]{FFFDE7}
		$ \boldsymbol{ \ldots }$ & 
		\cellcolor[HTML]{FFFDE7}
		$\boldsymbol{\nabla f(\vx_{m - 1})}$ &
		\cellcolor[HTML]{FFFDE7}
		$\boldsymbol{\nabla f(\vx_{m})}$ 
		\\[-\arrayrulewidth]
		\hhline{-----}
	\end{tabular}
	\caption{\emph{$m$-lazy second-order algorithms from \citet{doikov2023second}:} compute a new Hessian once per $m$ iterations.}
    \label{fig:lazy-hess}
	\centering
\end{figure}
\begin{dfn}[$m$-lazy second-order algorithm] \label{dfn:alg-class}
An $m$-lazy second-order algorithm \texttt{A} operating on $f : \sR^d \rightarrow \sR$ consists of a sequence of measurable mappings $\{ \texttt{A}_t \}_{t=1}^T$ of the form (\ref{eq:alg-class}), where
the oracle responses satisfy 
\begin{align*}
 \sO_t(\vx) =
 \begin{cases}
 (\nabla f(\vx), \nabla^2 f(\vx)), &   t = m, 2m , \cdots, m \lfloor T/m \rfloor;\\
 \nabla f(\vx), &  \text{otherwise}.
 \end{cases}
\end{align*}
\end{dfn}

 If $m=1$, the above algorithm class reduces to classical second-order algorithms, for which the zero-chain construction \citep{nesterov2018lectures,arjevani2019oracle} already gives a tight lower bound. The zero-chains refer to chain-like functions whose derivatives at a point with $t$ non-zero entries only have non-zero values up to the $(t+1)$-th entry. The optimal second-order zero-chain \citep{arjevani2019oracle} for $m=1$ is the following function:
\begin{align*}
    f_{2T}(\vx) = \lambda \vert a - \langle \vx, \vv_1 \rangle \vert^3 + \lambda \sum_{t=1}^{2T-1} \vert \langle \vx, \vv_{t+1} - \vv_t \rangle \vert^3,
\end{align*}
where $\{\vv_t \}_{t=1}^{2T}
$ is an orthonormal vector set. Although this instance can give the optimal lower bound of $\Omega(\epsilon^{-2/7})$ for classical second-order methods when $m=1$, it can not distinguish gradient and Hessian information and thus fails to reveal the $m$ dependency for the more general cases when $m \ge 1$.

 To overcome this challenge, we introduce a new block zero-chain of the following form:
\begin{align} \label{eq:block-zc}
\begin{split}
    f_{2T}(\vx) =& \lambda \left( (a-\langle \vx, \vv_1 \rangle)^2 + \sum_{t=1}^{m-1} (\langle \vx, \vv_{t+1} - \vv_t \rangle)^2  \right)^{3/2} \\
    &+ \lambda \sum_{s=1}^{\lfloor 2T/m \rfloor-1} \left( \sum_{t=sm}^{(s+1)m-1} (\langle \vx, \vv_{t+1} - \vv_t \rangle)^2  \right)^{3/2} + \lambda \left( \sum_{t=m \lfloor 2T/m \rfloor}^{2T-1} (\langle \vx, \vv_{t+1} - \vv_t \rangle)^2  \right)^{3/2}.
\end{split}
\end{align}
For this instance, each gradient call can discover one coordinate, while each Hessian call can discover all the $m$ coordinates in a block. This allows us to prove lower bounds for any integer $1 \le m \le T$. We formally present this block zero-chain property in the following lemma.

\begin{lem} \label{lem:block-zero-chain}
For any $m$-lazy second-order algorithm \texttt{A} operating on the instance $f_{2T}:\sR^d \rightarrow \sR$ defined in Eq.~(\ref{eq:block-zc}) with no more than $T$ iterations, for $d \ge 3T$, there exists an adversarial strategy to choose the vector set $\{ \vv_t \}_{t=1}^{2T}$, such that
    \begin{align} \label{eq:ortho-vectors}
    \langle \vx_T,\vv_i\rangle =0,
    \qquad
    \forall i=T+1,\ldots,2T .
\end{align}
\end{lem}

It means that only the first $T$ rotated coordinates of the algorithm output $\vx_T$ can be nonzero, and therefore its objective value is given by the truncated function $f_T$. We then use this property to give a lower bound on the suboptimality gap.

\begin{lem} \label{lem:subgap}
If Eq. (\ref{eq:ortho-vectors}) holds, we have
\begin{align} \label{eq:lb-suboptimality-gap}
    f_{2T}(\vx_T) \ge \frac{\lambda a^3}{m^{3/2} \left( \lfloor T/m \rfloor +1  \right)^2}.
\end{align}
\end{lem}

Finally, it suffices to set the parameters $\lambda$ and $a$ such that the function is convex, has $L$-Lipschitz continuous Hessians, and has a minimizer within distance $R$ of the origin, which yields the final lower bound.

\begin{thm} \label{thm:lower-bound-convex}
Given integers $1\le m\le T$ and any parameters $L>0$ and $D >0$, setting the parameters 
$\lambda =L/96$ and $a = D / \sqrt{2 T}$ for the instance  $f_{2T}:\sR^d \rightarrow \sR$ defined in Eq.~(\ref{eq:block-zc}) gives a function that has $L$-Lipschitz continuous Hessians and satisfies $\Vert \vx^* \Vert \le D$, such that,
any $m$-lazy second-order algorithm \texttt{A} with $T$ iterations,
operating on $f_{2T}$ can not find an $\epsilon$-solution unless $T = \Omega( m^{1/7} (L D^3  / \epsilon)^{2/7} )$.
\end{thm}

Note that an additional lower bound of $\Omega(m)$ is trivial by setting a simple quadratic function \citep[Section 2.1.4]{nesterov2018lectures}. The function can be rescaled with a large constant $\lambda$ such that no algorithm can obtain an approximate minimizer until it makes a Hessian evaluation in the $m$th iteration. Setting a sufficiently large $\lambda$ remains the Hessian Lipschitzness unchanged, such that the worst-case function is always valid. Consequently, we can obtain the final lower bound of $\Omega(m+m^{1/7} \epsilon^{-2/7} )$.

\subsection{Lower Bound for General Second-Order Algorithms} \label{subsec:lb-general}

Actually, our lower bound also holds for more general second-order algorithms that query $n_1$ gradients and $n_2$ Hessians when $n_2 \le n_1$. The formal definition for the algorithm class is as follows.

\begin{dfn}[General second-order algorithm]
\label{dfn:det-alg-class}
An deterministic second-order algorithm with unbalanced Hessian and gradient evaluations \texttt{A} operating on $f : \sR^d \rightarrow \sR$ consists of a sequence of measurable mappings $\{ \texttt{A}_t \}_{t=1}^T$ such that $\texttt{A}_t$ takes in the first $t$ oracle responses to produce the next iterate $\vx_t$,  \textit{i.e.}, \texttt{A} generates a sequence $\{\vx_t\}_{t=1}^{T}$ such that
\begin{align} \label{eq:alg-class}
    \vx_t = \texttt{A}_t ( \sO_1(\vx_0), \sO_2(\vx_1), \cdots, \sO_{t}(\vx_{t-1}) ), \quad t \ge 1; \qquad \vx_0 = \vzero,
\end{align}
where the oracle responses depend on the type $\phi_t \in \{ 1,2\}$:
\[
\sO_t(\vx) =
 \begin{cases}
\nabla^2 f(\vx) , & \text{if} \quad \phi_t = 1; \\
 \nabla f(\vx), & \text{if} \quad \phi_t = 2.
 \end{cases}
\]
We let $n_1$ and $n_2$ be the total number of $\phi_t = 1$ and $\phi_t=2$, respectively.
\end{dfn}



Using a simple simulation argument, we can show that the above algorithm class is equivalent to the $m$-lazy second-order algorithm class for $m = \lceil n_1/n_2 \rceil$ up to constant factors in the total number of iterations. Therefore, our lower bound in Theorem \ref{thm:lower-bound-convex} also holds for this more general algorithm class.

\begin{lem}\label{lem:upper-schedule-simulation}
For $n_2 \in [0,n_1]$, any general second-order algorithm $\texttt{A}$ using $n_1$ gradient queries and $n_2$ Hessian
queries can be simulated by an $m$-lazy algorithm $\texttt{A}'$ in
$n_1+m n_2$ iterations.
\end{lem}


\begin{proof}
When a gradient is needed for $\texttt{A}$, query it at the next iteration and discard the extra Hessian information if that
iteration is a snapshot. When a Hessian at a chosen point is needed for $\texttt{A}$, repeat queries at that point
until the next snapshot and discard the intervening responses. The
waiting time is fewer than $m$ iterations. Therefore, the total number of iterations of $\texttt{A}'$ is at most $n_1+m n_2$.
\end{proof}

\renewcommand{\algorithmiccomment}[1]{\hfill $\triangleright$ #1}
\section{Upper Complexity Bound}\label{sec:upper-complexity}

In this section, we construct an $m$-lazy second-order method in Algorithm \ref{alg:A-LEN} whose complexity nearly matches the preceding lower bound. The main theorem is stated as follows.

\begin{thm}\label{thm:ALEN}
Under Assumptions~\ref{asm:convex},~\ref{asm:Hessian-Lip}, and~\ref{asm:bounded-minimizer},
setting the parameters
\begin{align} \label{eq:param-alen}
    \gamma = L / m^3, \quad T = \gO( ( \gamma D^3 / \epsilon )^{2/7} ), \quad S = \gO(\log m), \quad N = \tilde \gO(m),
\end{align}
then Algorithm \ref{alg:A-LEN}
returns an $\epsilon$-solution in total gradient evaluations of
\begin{equation}\label{eq:upper-main}
\mathcal{O}\!\left(
\left(m+m^{1/7}\left(\frac{L D^3}{\epsilon}\right)^{2/7}\right)
\log^{11/3} m \right),
\end{equation}
where $D = \Vert \vx_0 - \vx^* \Vert$ is the initial distance to an arbitrary minimizer $\vx^*$.
\end{thm}

In a nutshell, our method combines optimal A-NPE (also known as search-free MS acceleration) by \citet{carmon2022optimal} 
with a new implementation of its inner cubic proximal oracle. The main ingredient is a preconditioned AGD with segment search that reduces the suboptimality gap of the subproblem by a constant factor using one Hessian and
$\widetilde{\mathcal{O}}(\kappa^{1/3})$ gradients, where $\kappa$ is the
second-order condition number of the subproblem. We obtain this
rate by minimizing a weighted sum of extrapolated objective.
The extrapolation controls curvature on the sublevel set visited by the
method, and a finite gradient-only line search keeps the descent iterates
inside that set. The detailed introduction to our algorithm is as follows.

\begin{algorithm*}[t]  
\caption{\texttt{Search-Free-MS-Acceleration}$(f, \vx_0, \gamma, T, S,N) $
}\label{alg:A-LEN}
\begin{algorithmic}[1] 
\STATE $\vv_0 = \vx_0$, $\bar \vx_0 = \vx_0$, $A_0 = 0$ \\
\STATE $\vy_0 = \texttt{Cubic-Proximal}(f,\vx_0,\gamma,S,N)$
\STATE If $\vy_0=\vx_0$, return $\vx_0$.
\STATE Set $\lambda_0' = \lambda_0= 1/(\gamma \Vert \vy_0 - \vx_0 \Vert) $. \\
\FOR{$t = 0,\cdots, T-1$}  
\STATE \quad $a_{t+1}' =   \lambda_{t}'(1 + \sqrt{1 + 4  A_t/ \lambda_{t}'}) / 2$, ~ $A_{t+1}' = A_t + a_{t+1}'$,~ $\bar \vx_t = (A_t \vx_t + a_{t+1}' \vv_t)/{A_{t+1}'} $ \label{line:momentum} \\
\STATE \quad \textbf{if} $t > 0 $ \textbf{then} 
\STATE \quad \quad $\vy_t = \texttt{Cubic-Proximal}(f,\bar{\vx}_t,\gamma,S,N)$
\STATE \quad \quad If $\vy_t=\bar{\vx}_t$, return $\bar{\vx}_t$.
\STATE \quad \quad Set $\lambda_t = 1/(\gamma \Vert \vy_t - \bar \vx_t \Vert)$. \\
\STATE \quad \textbf{end if} \\
\STATE \quad  \textbf{if} $\lambda_{t} \ge  \lambda_{t}'$ \textbf{then} \label{line:begin-seacg-lambda} \\
\STATE \quad \quad $a_{t+1} = a_{t+1}' $, $A_{t+1} = A_{t+1}'$, ~ $\vx_{t+1} = \vy_{t} $, ~ $\lambda_{t+1}' = 2 \lambda_{t}' $\\
\STATE \quad \textbf{else} \\
\STATE \quad \quad $\beta_{t} = \lambda_{t}/\lambda_{t}'$,~ $a_{t+1} = \beta_{t} a_{t+1}'$, ~ $A_{t+1} = A_t + a_{t+1}$ \\
\STATE \quad \quad $ \vx_{t+1} = ((1- \beta_{t}) A_t \vx_t + \beta_{t} A_{t+1}' \vy_{t}) / A_{t+1}$,~ $\lambda_{t+1}' = \lambda_{t}'/2$ \\
\STATE \quad \textbf{end if} \label{line:end-search-lambda} \\
\STATE \quad $\vv_{t+1} = \vv_t - a_{t+1} \nabla f(\vy_{t})$ \label{line:extra-ALEN} \\
\ENDFOR \\
\RETURN $\vx_T$ 
\end{algorithmic}
\end{algorithm*}

\subsection{Reduction to a Cubic Proximal Oracle}


In the first step, we leverage the result in \citet{carmon2022optimal} to reduce the task of designing an efficient second-order algorithm with optimal dependence on $\epsilon$ for convex optimization to the task of finding an approximate stationary point of the uniformly convex proximal functions:
\begin{equation}\label{eq:upper-cubic-objective}
f_{\vx,\gamma}(\vy)=f(\vy)+\frac{\gamma}{3}\|\vy-\vx\|^3,
\end{equation}
where $\vx \in \sR^d$ is the proximal center and $\gamma>0$ is the regularization. This function has a unique minimizer, we denote it as $\vy_{\vx,\gamma}^* := \arg \min_{\vy \in \sR^d} f_{\vx, \gamma}(\vy)$. We also denote its optimal value by $f_{\vx,\gamma}^* := f_{\vx,\gamma}(\vy_{\vx,\gamma}^*)$. It can be observed that, the optimal second-order method by \citet{carmon2022optimal} essentially only relies on the following approximate cubic proximal oracle to achieve the optimal $\epsilon^{-2/7}$ dependency.

\begin{cond} \label{cond:cubic-proxi}
The cubic proximal oracle $\vy = \texttt{Cubic-Proximal}(f,\vx,\gamma,S)$ always returns an output $\vy \in \sR^d$ satisfies the following Monteiro--Svaiter (MS) \citep{monteiro2012iteration,monteiro2013accelerated} condition
\begin{equation}\label{eq:upper-ms-condition}
\|\nabla f_{\vx,\gamma}(\vy)\|
\le \frac{\gamma}{32} \|\vy-\vx\|^2.
\end{equation}
\end{cond}

\begin{lem}[{\citet[Lemma 5.1]{chen2026faster}}] \label{lem:A-NPE}
Under Assumption \ref{asm:convex} and \ref{asm:bounded-minimizer}, if the \texttt{Cubic-Proximal} oracle always return an output satisfying Condition \ref{cond:cubic-proxi}, then running Algorithm \ref{alg:A-LEN} guarantees $f(\vx_T) - f^* \le \epsilon$ in 
\begin{equation}\label{eq:upper-outer-count}
T = \mathcal{O}\!\left(
\left(\frac{\gamma D^3}{\epsilon}\right)^{2/7}\right)
\end{equation}
iterations, where 
$D = \Vert \vx_0 - \vx^* \Vert$ is the initial distance to an arbitrary minimizer $\vx^*$.
\end{lem}
The above lemma is essentially from \citet[Theorem 1]{carmon2022optimal} using their search-free Monteiro--Svaiter acceleration algorithm \citep{monteiro2013accelerated,kovalev2022first,carmon2022optimal} whose procedure is presented in Algorithm \ref{alg:A-LEN}. We therefore focus
on designing a fast sub-solver using only logarithmically many Hessian queries to implement an approximate cubic proximal oracle that satisfies Condition~\ref{cond:cubic-proxi} efficiently, which can lead to a lower setting of $\gamma$ and finally result in a faster algorithm. To achieve this goal, we recall the properties of the proximal function as follows.

\begin{lem}[{\citet[Lemma 5.2]{chen2026faster} and \citet[Theorem 4.2.1]{nesterov2018lectures}} ] \label{lem:sub-prob-grad-dom}
Under Assumptions \ref{asm:convex} and~\ref{asm:Hessian-Lip}, the function $f_{\vx,\gamma}(\vy)$ defined in Eq. (\ref{eq:upper-cubic-objective}) satisfies:
\begin{enumerate}
    \item $f_{\vx,\gamma}(\vy)$ has $(L+2 \gamma)$-Lipschitz continuous Hessians.
    \item For any $\vy \in \sR^d$,
    \begin{align}
        \Vert \nabla f_{\vx, \gamma}(\vy) \Vert^{3/2} \ge \frac{3}{2} \sqrt{\frac{\gamma}{2}} (f_{\vx, \gamma}(\vy) - f_{\vx,\gamma}^*)
        \geq \frac{1}{2}\left( \frac{\gamma}{2} \Vert \vy - \vy_{\vx,\gamma}^* \Vert^2 \right)^{3/2}.\label{eq:cubic-grad-uniform}
    \end{align}
\end{enumerate}
\end{lem}
In light of the above lemma, we use \(H:= L+2 \gamma\) and \( \kappa:= H / \gamma \ge 2\) to denote the Hessian Lipschitz constant and the second-order condition number \citep[Eq. (4.2.22)]{nesterov2018lectures} of $f_{\vx,\gamma}$, respectively.

\begin{algorithm}[t]
\caption{\texttt{Cubic-Proximal}$(f,\vx, \gamma, S,N)$}
\label{alg:MS-solver}
\begin{algorithmic}[1]
\STATE Set $H=L+2\gamma$ and $\kappa=H/\gamma$.
\STATE Set $\vz_0=\mathcal{C}_H(\vx)$ and $\vy_0=\mathcal{C}_H(\vz_0)$.
\STATE Set $\Delta_0=\frac{2\sqrt2}{3\sqrt\gamma}
\|\nabla f_{\vx,\gamma}(\vy_0)\|^{3/2}$.
\FOR{$s=0,\ldots,S-1$}
\STATE $\vy_{s+1}=\texttt{PAGD-Seg}(f_{\vx,\gamma}, \vy_s, \Delta_s, H , \gamma,N)$
\STATE $\Delta_{s+1} = 3\Delta_s/4$
\ENDFOR
\STATE \textbf{return} $\hat \vy=\mathcal{C}_H(\vy_S)$.
\end{algorithmic}
\end{algorithm}

\subsection{Reduction From Making Gradient Small to Reducing the Gap} 
Condition \ref{cond:cubic-proxi} requires a subroutine to efficiently find an approximate stationary point of $f_{\vx,\gamma}$. However, the known guarantee for convex optimization typically focuses on reducing the suboptimality gap. To address this mismatch in the convergence criteria of the subroutine, we follow \citet{chen2026faster} to take several additional cubic regularized Newton (CRN) step \citep{nesterov2006cubic,nesterov2018lectures} to reduce the implementation of Condition \ref{cond:cubic-proxi} to reducing the suboptimality gap. The formal definition of the CRN step and its key properties are stated as follows:
\begin{equation}\label{eq:upper-cubic-step}
\mathcal{C}_H(\vz)=\vz+\mathop{\rm argmin}_{\mathbf{s} \in \mathbb{R}^d}
\langle\nabla f_{\vx,\gamma}(\vz),\mathbf{s}\rangle
+\frac12\mathbf{s}^{\top}\nabla^2f_{\vx,\gamma}(\vz)\mathbf{s}
+\frac H3\|\mathbf{s}\|^3.
\end{equation}

\begin{lem}[{\citet[(4.2.29) and (4.2.36)]{nesterov2018lectures}}]\label{lem:crn}
Under Assumptions~\ref{asm:convex} and~\ref{asm:Hessian-Lip}, set
$H=L+2\gamma$. For $\vy=\mathcal{C}_H(\vz)$, we have the following:
\begin{enumerate}
    \item $\|\nabla f_{\vx,\gamma}(\vy)\|^{3/2}
    \le 3\sqrt H\bigl(f_{\vx,\gamma}(\vz)-f_{\vx,\gamma}^*\bigr)$.
    \item $f_{\vx,\gamma}(\vy)-f_{\vx,\gamma}^*
    \le H\|\vz-\vy_{\vx,\gamma}^*\|^3/3$.
\end{enumerate}
\end{lem}

We implement Condition \ref{cond:cubic-proxi} by Algorithm \ref{alg:MS-solver}. This algorithm consists of $S$ epochs, and the CRN steps are only conducted at the first and last epochs, while the remaining epochs apply a subroutine using lazy second-order oracles to reduce the suboptimality gap. We now show that Condition \ref{cond:cubic-proxi} can be established by Algorithm \ref{alg:MS-solver} in $S = \gO(\log \kappa)$ iterations once the \texttt{PAGD-Seg} subroutine satisfied the following condition.


\begin{cond} \label{cond:gapreduce}
Under Assumptions \ref{asm:convex} and \ref{asm:Hessian-Lip}, for any $\Delta_s>0$ and $\vy_s$ satisfying
$f_{\vx,\gamma}(\vy_s)-f_{\vx,\gamma}^*\le\Delta_s$,
the subprocedure $\vy_{s+1}=\texttt{PAGD-Seg}
(f_{\vx,\gamma},\vy_s,\Delta_s,H,\gamma,N)$ returns a point satisfying
\[
f_{\vx,\gamma}(\vy_{s+1})-f_{\vx,\gamma}^*
\le\Delta_{s+1}:=\frac34\Delta_s.
\]
\end{cond}

\begin{lem}\label{lem:reduce-gap}
Under Assumptions~\ref{asm:convex} and~\ref{asm:Hessian-Lip}, let
$\gamma>0$, $H=L+2\gamma$, and $\kappa=H/\gamma$.
If the \texttt{PAGD-Seg} subprocedure always ensures Condition~\ref{cond:gapreduce}, then
Algorithm~\ref{alg:MS-solver}, with
\[
S=\left\lceil\frac{\log(50\cdot384^{3/2}\kappa^2)}{\log(4/3)}\right\rceil,
\]
returns $\hat\vy$ satisfying Condition~\ref{cond:cubic-proxi}, \textit{i.e.}, $ \| \nabla f_{\vx,\gamma}(\hat \vy)\| \le \gamma\| \hat \vy - \vx \|^2/32 $.
\end{lem}
Therefore, applying the above reductions, it remains to specify the subroutine that can satisfy Condition~\ref{cond:gapreduce}. Naive AGD~\citep{nesterov1983method} is not sufficient since it only has a sublinear convergence rate. To obtain a fast method, we develop the preconditioned accelerated gradient descent with segment search 
(\texttt{PAGD-Seg)} to achieve an exponential convergence rate. At a high level, we use the Hessian snapshot \citep{doikov2023second,chayti2023unified,adil2025balancing} as a preconditioner to turn the objective strongly convex locally. To better understand the \texttt{PAGD-Seg} algorithm, we introduce some additional properties of the proximal function.


\subsection{Curvature Control by an Extrapolated Objective}

 To facilitate the analysis, we introduce a technical lemma that controls the distance from a snapshot point $\vy$ to the minimizer $\vy_{\vx,\gamma}^*$ in the geometry relative to its Hessian by the suboptimality gap $f_{\vx,\gamma}(\vy) - f_{\vx,\gamma}^*$.

\begin{lem}\label{lem:upper-snapshot-distance}
Under Assumptions~\ref{asm:convex} and \ref{asm:Hessian-Lip}, for any $\vy \in \sR^d$,
\begin{equation}\label{eq:upper-snapshot-distance}
(\vy - \vy_{\vx,\gamma}^*)^{\top} \nabla^2 f_{\vx,\gamma}(\vy) (\vy - \vy_{\vx,\gamma}^*) 
\le 6(1+\kappa^{2/3})  (f_{\vx,\gamma}(\vy) - f_{\vx,\gamma}^*).
\end{equation}
\end{lem}

\begin{proof}
Let $\vd  = \vy - \vy_{\vx,\gamma}^*$. The claim is immediate when $\vd=0$.
Otherwise, by stationarity at $\vy_{\vx,\gamma}^*$ and Hessian $H$-Lipschitzness  we obtain
\begin{align*}
    f_{\vx,\gamma}(\vy) - f_{\vx,\gamma}^*  &=\int_0^1 t \vd^\top \nabla^2 f_{\vx,\gamma}(\vy - t \vd) \vd {\rm d} t\\
    &\geq \int_0^1 t \max\{0,\vd^\top (\nabla^2 f_{\vx,\gamma}(\vy) - Ht\Vert\vd\Vert\mI_d) \vd\} {\rm d} t\\
    &=
    \begin{cases} 
        \dfrac{1}{2}\vd^\top \nabla^2 f_{\vx,\gamma}(\vy) \vd - \dfrac{H}{3}\Vert\vd\Vert^3, &\vd^\top \nabla^2 f_{\vx,\gamma}(\vy) \vd\geq H\Vert\vd\Vert^3;\\[2mm]
        \dfrac{1}{6}(\vd^\top \nabla^2 f_{\vx,\gamma}(\vy) \vd)^3/(H\Vert\vd\Vert^3)^2, &\vd^\top \nabla^2 f_{\vx,\gamma}(\vy) \vd < H\Vert\vd\Vert^3.
    \end{cases}
\end{align*}
If $\vd^\top \nabla^2 f_{\vx,\gamma}(\vy) \vd\geq H\Vert\vd\Vert^3$, we have
\begin{align*}
    f_{\vx,\gamma}(\vy) - f_{\vx,\gamma}^*  &\geq \frac{1}{2}\vd^\top \nabla^2 f_{\vx,\gamma}(\vy) \vd - \frac{H}{3}\Vert\vd\Vert^3 \geq \frac{1}{6}\vd^\top \nabla^2 f_{\vx,\gamma}(\vy) \vd.
\end{align*}
If $\vd^\top \nabla^2 f_{\vx,\gamma}(\vy) \vd < H\Vert\vd\Vert^3$, by item 2 of Lemma \ref{lem:sub-prob-grad-dom}, we have $ H \| \vd \|^3 \le 6\kappa(f_{\vx,\gamma}(\vy) - f_{\vx,\gamma}^*) $. Therefore, 
\begin{align*}
    f_{\vx,\gamma}(\vy) - f_{\vx,\gamma}^*  &\geq \frac{1}{6}(\vd^\top \nabla^2 f_{\vx,\gamma}(\vy) \vd)^3/(H\Vert\vd\Vert^3)^2\\
    &\geq \frac{1}{6}(\vd^\top \nabla^2 f_{\vx,\gamma}(\vy) \vd)^3/(6\kappa(f_{\vx,\gamma}(\vy) - f_{\vx,\gamma}^*))^2,
\end{align*}
so $f_{\vx,\gamma}(\vy) - f_{\vx,\gamma}^* \geq (1/{6\kappa^{2/3}})\vd^\top \nabla^2 f_{\vx,\gamma}(\vy) \vd$. In both cases, we have inequality (\ref{eq:upper-snapshot-distance}) holds.
\end{proof}
Intuitively, the above lemma suggests that an accelerated gradient descent (AGD) \citep{nesterov1983method,nesterov2018lectures} with the preconditioner given by $\nabla^2 f_{\vx,\gamma}(\vy)$ solves the cubic proximal subproblem in the number of iterations of 
\[
\left(\frac{(\vy - \vy_{\vx,\gamma}^*)^{\top} \nabla^2 f_{\vx,\gamma}(\vy) (\vy - \vy_{\vx,\gamma}^*) }{ f_{\vx,\gamma}(\vy) - f_{\vx,\gamma}^*}\right)^{1/2} = \gO( \kappa^{1/3}).
\]
However, directly applying (preconditioned) AGD to the original proximal function may be unstable, since the momentum may lead to the arrival of some point that has unbounded Hessians. For the purpose of stability, we introduce a novel extrapolated objective whose Hessian is stable in a sublevel set. Let
\begin{equation}\label{eq:upper-extrapolation-parameters}
\lambda=\frac1{64},\quad
J=\left\lceil\frac{\log\kappa}{3\log64}\right\rceil\geq1,
\quad c=1+\frac1J,
\end{equation}
Our extrapolated objective $g$ is defined as follows:
\begin{equation}\label{eq:upper-extrapolated-objective}
g(\vv)=\sum_{i=0}^J\lambda^if_{\vx,\gamma}(\vv_i),\quad {\rm where} \quad \vv_i=\vy+c^i(\vv-\vy).
\end{equation}
The following lemma gives the core property of the extrapolated objective, which bounds its Hessian $\nabla^2 g(\vv)$ in the sublevel set
using the Hessian of the original objective at the snapshot point $\nabla^2 f_{\vx,\gamma} (\vy)$.


\begin{lem}\label{lem:upper-extrapolated-curvature}
Under Assumptions \ref{asm:convex} and \ref{asm:Hessian-Lip}, for $\Delta>0 $ satisfies $ f_{\vx,\gamma}(\vy)-f_{\vx,\gamma}^* \le \Delta$, let
\begin{equation}\label{eq:upper-metric}
\mB=2 (J+1) \nabla^2f_{\vx,\gamma}(\vy)+100 (J+1)^{4/3}(H^2 \Delta)^{1/3} \mI_d.
\end{equation}
Then
\begin{equation}\label{eq:upper-sublevel-smoothness}
\nabla^2 g(\vv)\preceq \mB, \qquad \forall \vv \in \{ \vv' \in \sR^d: g(\vv') \le g(\vy) \}.
\end{equation}
\end{lem}
\begin{proof}
Recalling that $\lambda =  1/64$, we define $W = \sum_{i=0}^J \lambda^i \le 64/63$. Suppose $g(\vv)\le g(\vy)$, so
\begin{align}\label{eq:upper-individual-gaps}
    0\le f_{\vx,\gamma}(\vv_i)-f_{\vx,\gamma}^*
\le\lambda^{-i}\bigl(g(\vv)-Wf_{\vx,\gamma}^*\bigr)
\le\lambda^{-i}\bigl(g(\vy)-Wf_{\vx,\gamma}^*\bigr)
\le W\Delta\lambda^{-i},
\end{align}
For $\Delta>0$ satisfying $f_{\vx,\gamma}(\vy)-f_{\vx,\gamma}^*\le \Delta$, we know from Lemma \ref{lem:sub-prob-grad-dom}(2), $\|\vy-\vy_{\vx,\gamma}^*\|$ is bounded by $R\triangleq(6\Delta/\gamma)^{1/3}$. Then, applying the second inequality of Lemma \ref{lem:sub-prob-grad-dom}(2),
we have
\begin{equation}\label{eq:upper-individual-distance}
\|\vv-\vy\|\le\|\vv-\vy_{\vx,\gamma}^*\|+\|\vy-\vy_{\vx,\gamma}^*\|\le(1+W^{1/3})R\le3R.
\end{equation}
For $i<J$, the adjacent extrapolated points satisfy the relationship 
$\vv_i=  (\vy+ J \vv_{i+1}) / (J+1)$.
For every unit vector $\mathbf{e}$ and $t>0$, convexity gives
\[
f_{\vx,\gamma}(\vv_i\pm t\mathbf{e})
\le  \frac{1}{J+1} \left(f_{\vx,\gamma}(\vy \pm (J+1) t\mathbf{e})
+ J f_{\vx,\gamma}(\vv_{i+1}) \right).
\]
Adding these inequalities and applying the second-order Taylor bounds
\begin{align*}
    \left|f_{\vx,\gamma}(\vu+\vh) + f_{\vx,\gamma}(\vu-\vh) - 2f_{\vx,\gamma}(\vu) - \vh^\top \nabla^2f_{\vx,\gamma} (\vu)\vh\right|
    \leq  \frac{H}{3} \Vert\vh\Vert^3
\end{align*}
at $\vy$ and $\vv_i$ yields
\begin{align*}
    &2f_{\vx,\gamma}(\vv_i) + t^2\mathbf{e}^{\top}\nabla^2f_{\vx,\gamma}(\vv_i)\mathbf{e}-\frac{H}{3}t^3
    \\
    \leq & \frac{2}{J+1}f_{\vx,\gamma}(\vy) + (J+1)t^2\mathbf{e}^{\top} \nabla^2f_{\vx,\gamma}(\vy) \mathbf{e} + \frac{H}{3}(J+1)^2t^3 + \frac{2J}{J+1}f_{\vx,\gamma}(\vv_{i+1}).
\end{align*}
Next, we can use inequality (\ref{eq:upper-individual-gaps}) to obtain that
\begin{align*}
&\mathbf{e}^{\top}\nabla^2f_{\vx,\gamma}(\vv_i)\mathbf{e} \\
&\le (J+1) \mathbf{e}^{\top}  \nabla^2f_{\vx,\gamma}(\vy) \mathbf{e}+\frac{H(1+(J+1)^2)}3t
+\frac{2}{ t^2 (J+1)}(f_{\vx,\gamma}(\vy)+J f_{\vx,\gamma}(\vv_{i+1})-(J+1)f_{\vx,\gamma}(\vv_i))\\
&\le (J+1)\mathbf{e}^{\top}  \nabla^2f_{\vx,\gamma}(\vy) \mathbf{e}+\frac{H(1+(J+1)^2)}3t
+\frac{2}{t^2 (J+1)}(f_{\vx,\gamma}(\vy)-f_{\vx,\gamma}^*+J(f_{\vx,\gamma}(\vv_{i+1})-f_{\vx,\gamma}^*))\\
&\le (J+1) \mathbf{e}^{\top}  \nabla^2f_{\vx,\gamma}(\vy) \mathbf{e}+\frac{H(1+(J+1)^2)}3t
+\frac{2}{t^2 (J+1)}(\Delta+ J W \Delta \lambda^{-(i+1)})\\
&\le (J+1)\mathbf{e}^{\top}  \nabla^2f_{\vx,\gamma}(\vy) \mathbf{e}+\frac{H(1+(J+1)^2)}3t+\frac{2W \Delta \lambda^{-(i+1)}}{t^2}.
\end{align*}
Choose $t=(\Delta \lambda^{-(i+1)}/(H (J+1)^2))^{1/3}$ and write
$\delta=(H^2 \Delta )^{1/3}$. Since $J \ge 1$,
\begin{align} \label{eq:upper-adjacent-curvature}
\begin{split}
\nabla^2f_{\vx,\gamma}(\vv_i)
\preceq & (J+1) \nabla^2f_{\vx,\gamma}(\vy) +\left(2W+\frac{1+(J+1)^2}{3 (J+1)^2}\right) (J+1)^{4/3}H^{2/3} \Delta^{1/3}\lambda^{-(i+1)/3} \mI_d \\
\preceq & (J+1) \nabla^2f_{\vx,\gamma}(\vy)+\frac52 (J+1)^{4/3}\delta\lambda^{-(i+1)/3} \mI_d.
\end{split}    
\end{align}
For $c = (1+1/J)$, multiplying by $\lambda^ic^{2i}$ and summing over $i<J$, using Hessian Lipschitzness in Assumption~\ref{asm:Hessian-Lip} and inequality~(\ref{eq:upper-individual-distance}), we can deduce that
\begin{align*}
    \nabla^2 g(\vv)&=\sum_{i=0}^{J}\lambda^i c^{2i} \nabla^2f_{\vx,\gamma}(\vv_i)\\
    &\preceq \sum_{i=0}^{J-1} \lambda^i c^{2i} \left( (J+1) \nabla^2f_{\vx,\gamma}(\vy)+\frac52 (J+1)^{4/3}\delta\lambda^{-(i+1)/3}\mI_d\right)
    +\lambda^Jc^{2J}\nabla^2f_{\vx,\gamma}(\vv_J)\\
    &\preceq (J+1) \nabla^2f_{\vx,\gamma}(\vy)\sum_{i=0}^{J-1} c^{2i}\lambda^i+\frac52 (J+1)^{4/3}\delta \mI_d \sum_{i=0}^{J-1} c^{2i}\lambda^{(2i-1)/3}
    +\lambda^Jc^{2J}(\nabla^2f_{\vx,\gamma}(\vy)+H\Vert \vv_J-\vy\Vert \mI_d)\\
    &\preceq (J+1) \nabla^2f_{\vx,\gamma}(\vy) \sum_{i=0}^{J} c^{2i}\lambda^i+\frac52 (J+1)^{4/3}\delta \mI_d \sum_{i=0}^{J-1} c^{2i}\lambda^{(2i-1)/3}
    +3\lambda^Jc^{3J}HR \mI_d.
\end{align*}
Since
$\lambda c^2\le 1/16$ and $\lambda^{2/3}c^2\le 1/4$,
$\sum_{i=0}^{J} c^{2i}\lambda^i\leq {16}/{15}$ and $2.5\sum_{i=0}^{J-1} c^{2i}\lambda^{(2i-1)/3} \leq 40/3$. Also $c^J\le {\rm e}$ and $\lambda^J\le\kappa^{-1/3}$,
so
\begin{align*}
    3\lambda^Jc^{3J}HR\leq 3\kappa^{-1/3}{\rm e}^3H(6\Delta/\gamma)^{1/3}=3{\rm e}^3 6^{1/3}\delta<132 (H^2 \Delta)^{1/3}.
\end{align*}
Therefore,
\begin{align*}
    \nabla^2 g(\vv)\preceq\frac{16}{15}(J+1) \nabla^2f_{\vx,\gamma}(\vy)+\left(\frac{40}{3} (J+1)^{4/3}+132\right)(H^2 \Delta)^{1/3} \mI_d \preceq \mB.
\end{align*}
\end{proof}

\subsection{Preconditioned Accelerated Gradient Descent with Segment Search}

In this section, we leverage the properties of the extrapolated objective and propose a preconditioned accelerated gradient descent with segment search (\texttt{PAGD-Seg}) in Algorithm~\ref{alg:upper-gap-reduction} to implement Condition~\ref{cond:gapreduce}.

An important consequence of the matrix $\mB$ defined in Eq. (\ref{eq:upper-metric}) is the following fact, which indicates that there exists an approximate minimizer $\bar \vv \approx \arg \min_{\vv \in \sR^d} g(\vv)$ that satisfies the following properties.

\begin{lem} \label{lem:comparator}
Under Assumptions \ref{asm:convex} and \ref{asm:Hessian-Lip}, for $\Delta>0 $ satisfies $ f_{\vx,\gamma}(\vy)-f_{\vx,\gamma}^* \le \Delta$, the averaged point $\bar{\vv}=\vy+c^{-J}(\vy_{\vx,\gamma}^*-\vy)$ satisfies
\begin{enumerate}
    \item $g(\bar \vv)- \sum_{i=0}^J \lambda^i f_{\vx,\gamma}^*<0.643\Delta$;
    \item $\|\bar{\vv}-\vy\|_\mB^2 \le 90 (J+1)^{4/3}\kappa^{2/3} \Delta$.
\end{enumerate}


\end{lem}

\begin{proof}
For $c = (1+1/J)$ and $W = \sum_{i=0}^J \lambda^i$, all the points $\bar \vv_i = \vy+c^{i-J}(\vy_{\vx,\gamma}^*-\vy)$ lie on the segment $[\vy,\vy_{\vx,\gamma}^*]$. Since $c^{-J}\ge {\rm e}^{-1}$,
the convexity of $f_{\vx,\gamma}$ implies
$g(\bar{\vv})-Wf_{\vx,\gamma}^*
\le W(1-{\rm e}^{-1}) \Delta<0.643 \Delta$.
Finally, $c^{-J}\le1/2$ and Lemma~\ref{lem:upper-snapshot-distance} imply
\[
\|\bar{\vv}-\vy\|_\mB^2
= c^{-2J}\|\vy_{\vx,\gamma}^*-\vy \|_\mB^2
\le\bigl(6 (J+1)+25\cdot6^{2/3} (J+1)^{4/3}\bigr)\kappa^{2/3} \Delta
\le90 (J+1)^{4/3}\kappa^{2/3} \Delta.
\]
\end{proof}


The \texttt{PAGD-Seg} algorithm combines preconditioned accelerated gradient steps with a segment search
between the current iterate and the auxiliary sequence.
The search keeps the gradient steps within the sublevel set where
Lemma~\ref{lem:upper-extrapolated-curvature} applies, while controlling the error in the accelerated potential argument in standard AGD analysis \citep{nesterov2018lectures}.
We require the following segment-search guarantee.



\begin{cond}\label{cond:segment-search}
For any differentiable convex function $g$, points $\vx,\vv\in\sR^d$, and $\zeta>0$, the oracle
$\vy=\texttt{FiniteSegSearch}(g,\vx,\vv,\zeta)$ returns a point $\vy\in[\vx,\vv]$ satisfying
\[
g(\vy)\le g(\vx),\qquad
\langle\nabla g(\vy),\vy-\vv\rangle\le\zeta.
\]
\end{cond}

\begin{algorithm}[t]
\caption{\texttt{PAGD-Seg}$(f_{\vx,\gamma},\vy,\Delta,H,\gamma,N)$}
\label{alg:upper-gap-reduction}
\begin{algorithmic}[1]
\STATE Query $\nabla^2f_{\vx,\gamma}(\vy)$ and form $g,\mB$ using~\eqref{eq:upper-extrapolated-objective} and~\eqref{eq:upper-metric}.
\STATE Set $\vx_0=\vv_0=\vy$, $A_0=0$, and $\zeta=\Delta/100$.
\FOR{$k=0,\ldots,N-1$}
\STATE $\vy_k=\texttt{FiniteSegSearch}(g,\vx_k,\vv_k,\zeta)$ and $\vh_k=\nabla g(\vy_k)$.
\STATE Set $a_{k+1}=(1+\sqrt{1+4A_k})/2$ and $A_{k+1}=A_k+a_{k+1}$.
\STATE Update $\vx_{k+1}=\vy_k-\mB^{-1}\vh_k$ and $\vv_{k+1}=\vv_k-a_{k+1}\mB^{-1}\vh_k$.
\ENDFOR
\STATE \textbf{return} $\vx_N$.
\end{algorithmic}
\end{algorithm}

Let $\phi(t)=g(\vx+t ( \vv - \vy))$. Using the convexity of this one-dimensional function, we show in Section \ref{sec:segment-search} that the above condition can be fulfilled by finding $l \in [0,1]$ such that $- \zeta \le \phi'(l) \le 0$. Therefore, we can implement the above oracle using a simple bisection search. We now assume that this condition is satisfied and state the accelerated rate. achieved by the \texttt{PAGD-Seg} algorithm.

\begin{lem}\label{lem:upper-gap-reduction}
Under Assumptions~\ref{asm:convex} and~\ref{asm:Hessian-Lip}, and Condition~\ref{cond:segment-search}, Algorithm~\ref{alg:upper-gap-reduction} with
\[
\zeta={\Delta}/{100},\qquad
N=\left\lceil64(J+1)^{2/3}(H/\gamma)^{1/3}\right\rceil
\]
satisfies Condition~\ref{cond:gapreduce}, \textit{i.e.},  for any $\Delta>0$ and $\vy$ with $f_{\vx,\gamma}(\vy)-f_{\vx,\gamma}^*\le\Delta$, its output satisfies
\begin{equation}\label{eq:upper-gap-contraction}
f_{\vx,\gamma}(\vx_N)-f_{\vx,\gamma}^*\le 3\Delta/4.
\end{equation}
It uses one Hessian evaluation, $N$ segment-search calls, and $N$ additional gradient evaluations of $g$. With the implementation in Algorithm~\ref{alg:finite-seg-search}, the total number of gradient evaluations of $f_{\vx,\gamma}$ is $\mathcal{O}(\kappa^{1/3}\log^{8/3}\kappa)$.
\end{lem}

\begin{proof}
We first need a descent property in the sublevel set.
For any $\vy'$ with $g(\vy')\le g(\vy)$, put $\vh=\nabla g(\vy')$ and $\phi(t)=g(\vy'-t\mB^{-1}\vh)$ for $t\in[0,1]$. Whenever $\phi(t_0)\le\phi(0)$, convexity places the whole segment up to $t_0$ in the initial sublevel set. Lemma~\ref{lem:upper-extrapolated-curvature} then gives
\begin{align*}
\phi(t_0)-\phi(0)
&=-t_0\|\vh\|_{\mB^{-1}}^2
+\int_0^{t_0}(t_0-t)(\mB^{-1}\vh)^\top
\nabla^2g(\vy'-t\mB^{-1}\vh)(\mB^{-1}\vh)\,\mathrm{d}t\\
&\le-(t_0-t_0^2/2)\|\vh\|_{\mB^{-1}}^2.
\end{align*}
If $\phi(1)>\phi(0)$, then $\phi'(0)=-\|\vh\|_{\mB^{-1}}^2<0$ implies that some $t_0\in(0,1)$ satisfies $\phi(t_0)=\phi(0)$, contradicting this strict decrease. Thus $\phi(1)\le\phi(0)$, and taking $t_0=1$ yields
\begin{equation}\label{eq:upper-preconditioned-descent}
g(\vy'-\mB^{-1}\vh)\le g(\vy')-\|\vh\|_{\mB^{-1}}^2/2, \quad \forall \{ \vy' \in \sR^d: g(\vy') \le g(\vy) \}.
\end{equation}
Starting from $\vx_0=\vy$, Condition~\ref{cond:segment-search} and inequality (\ref{eq:upper-preconditioned-descent}) imply inductively that
\[
g(\vx_{k+1})\le g(\vy_k)\le g(\vx_k)\le g(\vy).
\]
For the approximate minimizer $\bar{\vv}$ in Lemma~\ref{lem:comparator}, define the Lyapunov function of AGD as
\[
E_k=A_k\bigl(g(\vx_k)-g(\bar{\vv})\bigr)
+\|\vv_k-\bar{\vv}\|_{\mB}^2/2.
\]
Therefore, the decrement of the Lyapunov function can be decomposed into two parts
\begin{align*}
E_{k+1} - E_k =& \underbrace{A_{k+1} (g(\vx_{k+1}) - g(\bar \vv)) - A_k (g(\vx_k) - g(\bar \vv))}_{\rm (I)} \\    
+& \underbrace{\|\vv_{k+1}-\bar{\vv}\|_{\mB}^2/2 - \|\vv_k-\bar{\vv}\|_{\mB}^2/2 }_{\rm (II)}.
\end{align*}

For part (I), we can use inequality (\ref{eq:upper-preconditioned-descent}) and the fact that $A_{k+1} = A_k + a_{k+1}$ to obtain
\begin{align*}
{\rm (I)} \le & A_{k+1} (g(\vy_k) - g(\bar \vv)) - \frac{a_{k+1}^2}{2} \|\vh_k \|_{\mB^{-1}}^2  - A_k(g(\vx_k) -g(\bar \vv)) \\
=& A_k (g(\vy_k) - g(\vx_k)) + a_{k+1} (g(\vy_k) - g(\bar{\vv})) - \frac{a_{k+1}^2}{2} \|\vh_k \|_{\mB^{-1}}^2.
\end{align*}
For part (II), substituting the update rule $\vv_{k+1} = \vv_k - a_{k+1} \mB^{-1} \vh_k$ into the quadratic term yields
\begin{align*}
{\rm (II)} =& \frac{1}{2} \|\vv_k - a_{k+1} \mB^{-1} \vh_k - \bar{\vv}\|_{\mB}^2 - \frac{1}{2} \|\vv_k - \bar{\vv}\|_{\mB}^2 = -a_{k+1} \langle \vh_k, \vv_k - \bar{\vv} \rangle + \frac{a_{k+1}^2}{2} \|\vh_k\|_{\mB^{-1}}^2.
\end{align*}
Combining (I) and (II), the squared-gradient term ${a_{k+1}^2} \|\vh_k\|_{\mB^{-1}}^2/2$ cancels out exactly, giving
\begin{align*}
E_{k+1} - E_k \le A_k \bigl(g(\vy_k) - g(\vx_k)\bigr) + a_{k+1} \bigl(g(\vy_k) - g(\bar{\vv}) - \langle \vh_k, \vv_k - \bar{\vv} \rangle\bigr).
\end{align*}
Condition~\ref{cond:segment-search} gives $g(\vy_k)\le g(\vx_k)$, while convexity gives $g(\vy_k)-g(\bar{\vv})\le\langle\vh_k,\vy_k-\bar{\vv}\rangle$. Since $A_k\ge0$, substituting these bounds and using the second inequality in Condition~\ref{cond:segment-search} yields
\begin{align*}
E_{k+1}-E_k
&\le a_{k+1}\bigl(\langle\vh_k,\vy_k-\bar{\vv}\rangle-\langle\vh_k,\vv_k-\bar{\vv}\rangle\bigr) =a_{k+1}\langle\vh_k,\vy_k-\vv_k\rangle
\le a_{k+1}\zeta.
\end{align*}
Since $A_0=0$, we have \(E_0=\|\vy-\bar{\vv}\|_{\mB}^2/2\). Therefore, telescoping the above inequality yields
\begin{equation}\label{eq:upper-potential-bound}
E_k\le \|\vy-\bar{\vv}\|_{\mB}^2/2 +A_k\zeta, \quad \forall 0\le k\le N.
\end{equation}
Note that we can prove $A_k \ge (k+1)^2/4$ for all $ k\ge 1$ by induction. The base case $A_1 =1$ is trivial. Assume $A_k \ge (k+1)^2/4$, then we have
\[
\sqrt{A_{k+1}} = \frac{1 + \sqrt{1 + 4A_k}}{2} 
\ge \frac{1 + \sqrt{1 + (k+1)^2}}{2} 
> \frac{1 + (k+1)}{2} 
= \frac{k+2}{2}.
\]
Squaring both sides gives \(A_{k+1} \ge (k+2)^2/{4}\), which completes the induction step. Thus, $A_k \ge {(k+1)^2}/{4}$ holds for all $k \ge 1$. Then, using Lemma~\ref{lem:comparator}(2), we obtain
\begin{equation}\label{eq:upper-accelerated-error}
g(\vx_N)-g(\bar{\vv})
\le\frac{\|\vy-\bar{\vv}\|_{\mB}^2}{2A_N}+\zeta
\le\left(\frac{180}{4096}+\frac1{100}\right)\Delta<0.054\Delta.
\end{equation}
Let $W=\sum_{i=0}^J\lambda^i$. Since every summand of $g(\vx_N)-Wf_{\vx,\gamma}^*$ is nonnegative and its $i=0$ summand is $f_{\vx,\gamma}(\vx_N)-f_{\vx,\gamma}^*$, Lemma~\ref{lem:comparator}(1) gives
\[
f_{\vx,\gamma}(\vx_N)-f_{\vx,\gamma}^*
\le g(\vx_N)-Wf_{\vx,\gamma}^*
<(0.643+0.054)\Delta< 3\Delta/4.
\]
This proves inequality~\ref{eq:upper-gap-contraction}. The algorithm forms $\mB$ with one Hessian query and makes $N$ segment-search calls and $N$ additional gradient queries to~$g$. Each gradient of $g$ uses $J+1$ gradients of $f_{\vx,\gamma}$, since
\[
\nabla g(\vv)=\sum_{i=0}^J\lambda^ic^i\nabla f_{\vx,\gamma}\bigl(\vy+c^i(\vv-\vy)\bigr).
\]
Lemma~\ref{lem:finite-segment-search} proves that each search uses $\mathcal{O}(\log\kappa)$ such gradients, yielding the stated total count.
\end{proof}

\subsection{Implementing the Segment Search Oracle}\label{sec:segment-search}

\begin{algorithm}[t]
\caption{\texttt{FiniteSegSearch}$(g,\vx,\vv,\zeta)$}
\label{alg:finite-seg-search}
\begin{algorithmic}[1]
\STATE Set $\vd=\vv-\vx$ and $\phi(t)=g(\vx+t\vd)$.
\IF{$\phi'(0)\ge0$}
\STATE \textbf{return} $\vx$.
\ENDIF
\IF{$\phi'(1)\le0$}
\STATE \textbf{return} $\vv$.
\ENDIF
\STATE Set $l=0$ and $r=1$.
\WHILE{$-\phi'(l)>\zeta$}
\STATE Set $t=(l+r)/2$ and query $\phi'(t)$.
\IF{$\phi'(t)=0$}
\STATE \textbf{return} $\vx+t\vd$.
\ELSIF{$\phi'(t)<0$}
\STATE Set $l=t$.
\ELSE
\STATE Set $r=t$.
\ENDIF
\ENDWHILE
\STATE \textbf{return} $\vx+l\vd$.
\end{algorithmic}
\end{algorithm}

We implement Condition~\ref{cond:segment-search} by binary search on the directional derivative, as presented in Algorithm \ref{alg:finite-seg-search}.

\paragraph{Finite segment search.}  For $\vd=\vv-\vx$, the restriction $\phi(t)=g(\vx+t\vd)$ is convex on $[0,1]$, so $\phi'(t)=\langle\vd,\nabla g(\vx+t\vd)\rangle$ is nondecreasing. If $\phi'(0)\ge0$ or $\phi'(1)\le0$, the corresponding endpoint satisfies the oracle condition; otherwise, the endpoint derivatives have opposite signs, and we can use a standard bisection procedure to locate a left endpoint $l \in [0,1]$ with $-\zeta\le\phi'(l)\le0$. 

It is not hard to see that the return value satisfies Condition \ref{cond:segment-search}: the monotonicity of $\phi'(t)$ implies that $\phi'(t)$ is negative on $[0,l]$, which implies that $g(\vx+ l \vd) \le g(\vx)$. Moreover, $\phi'(t)=\langle\vd,\nabla g(\vx+t\vd)\rangle$ implies that 
$\langle\nabla g(\vx+l\vd),\vx+l\vd-\vv\rangle=-(1-l)\phi'(l)\le\zeta$, These are precisely the two required inequalities in Condition~\ref{cond:segment-search}.


\paragraph{Number of Segment Search.}
It remains to bound the number of gradient calls in the segment search.
For $\rho=\max\{\|\vx_k-\vy\|,\|\vv_k-\vy\|\}$,
the global Hessian Lipschitz bound by Assumption \ref{asm:Hessian-Lip} gives, for every $\vw\in[\vx_k,\vv_k]$,
\begin{equation}\label{eq:upper-line-curvature}
\nabla^2g(\vw)
\preceq\left(\sum_{i=0}^J\lambda^ic^{2i}\right)\nabla^2f_{\vx,\gamma}(\vy)
+H\|\vw-\vy\|\left(\sum_{i=0}^J\lambda^ic^{3i}\right)\mI_d
\preceq2\nabla^2f_{\vx,\gamma}(\vy)+2H\rho\mI_d.
\end{equation}
Here $\lambda c^2\le1/16$ and $\lambda c^3\le1/8$. Since $g$ is twice continuously differentiable, $\phi'$ has a finite Lipschitz constant $M$ on the compact interval $[0,1]$, where
\begin{equation}\label{eq:upper-line-lipschitz}
M=(\vv_k-\vx_k)^\top(2 \nabla^2f_{\vx,\gamma}(\vy) +2H\rho\mI_d)(\vv_k-\vx_k).
\end{equation}
If $M=0$, then $\phi'$ is constant and the algorithm returns an endpoint after at most two gradient queries. Assume henceforth that $M>0$. After $j$ bisections the size of the bracket satisfies $-\phi'(l)\le M(r-l)=M2^{-j}$. Therefore, the bisection procedure terminates in 
at most $1+\lceil\log_2^+(M/\zeta)\rceil$ steps, with
$\log_2^+a=\max\{0,\log_2 a\}$ for $a>0$. Finally, we can complete the proof by showing that $M / \zeta = \gO(\kappa)$.



\begin{lem}\label{lem:finite-segment-search}
Under Assumptions~\ref{asm:convex} and~\ref{asm:Hessian-Lip}, let $\Delta>0$ and $f_{\vx,\gamma}(\vy)-f_{\vx,\gamma}^*\le\Delta$. For every call made by Algorithm~\ref{alg:upper-gap-reduction} with these inputs and $\zeta=\Delta/100$, Algorithm~\ref{alg:finite-seg-search} terminates and returns a point satisfying Condition~\ref{cond:segment-search}, using $\mathcal{O}(\log\kappa)$ gradient evaluations of $g$. Consequently, Algorithm~\ref{alg:upper-gap-reduction} uses a total of $\mathcal{O}(\kappa^{1/3}\log^{8/3}\kappa)$ gradient evaluations of $f_{\vx,\gamma}$ whenever it is called.
\end{lem}

Finally, combining all the above analyzes in this section, we can complete the proof of our main theorem for the near-optimal upper bound of $\tilde \gO(m+ m^{1/7} \epsilon^{-2/7})$ by appropriately choosing the parameters to balance the gradient and Hessian evaluations as follows.

\begin{proof}[Proof of Theorem~\ref{thm:ALEN}]
Set $\gamma=L/m^3$, so $\kappa=m^3+2$, and choose $S,N$
as in Lemmas~\ref{lem:reduce-gap} and~\ref{lem:upper-gap-reduction}.
These lemmas guarantee Condition~\ref{cond:cubic-proxi}.
By Lemma~\ref{lem:finite-segment-search}, each of the
$S=\mathcal{O}(\log\kappa)$ phases uses one Hessian evaluation
and $\mathcal{O}(\kappa^{1/3}\log^{8/3}\kappa)$ gradient evaluations.
Including the three CRN steps, each proximal call therefore uses
\[
 n_2 = \mathcal{O}(\log\kappa), \qquad n_1=\mathcal{O}(\kappa^{1/3}\log^{11/3}\kappa).
\]
Lemma~\ref{lem:upper-schedule-simulation} implements each such call
in $\mathcal{O}(n_1+m(n_2+1))
=\mathcal{O}(m\log^{11/3}m)$ iterations under the fixed
snapshot schedule. Multiplying this cost by the outer iteration
bound~\eqref{eq:upper-outer-count} gives
\[
\mathcal{O}\!\left(
m\log^{11/3} m
\left(1+\left(\frac{LD^3}{m^3\epsilon}\right)^{2/7}\right)
\right)
=
\mathcal{O}\!\left(
\left(m+m^{1/7}\left(\frac{LD^3}{\epsilon}\right)^{2/7}\right)
\log^{11/3}m
\right).
\]
The additional outer gradient queries are absorbed by this bound,
proving \eqref{eq:upper-main}.
\end{proof}

\section{Conclusion and Future Works}


In this paper, we study the oracle complexity of 
convex optimization using $m$-lazy second-order oracles. We prove a lower bound of $\Omega(m +m^{1/7} \epsilon^{-2/7})$ and a nearly matching upper bound of $\tilde \gO(m+m^{1/7} \epsilon^{-2/7})$.

Looking ahead, our work opens up several promising avenues for future research. 
On the theoretical side, while this paper resolves the convex setting up to logarithmic factors, the complexity for nonconvex settings remains an open, where the best-known upper and lower bounds \citep{chen2026faster} for finding an $\epsilon$-stationary point are $\gO(m+m^{1/3} \epsilon^{-3/2})$ and $\Omega(\epsilon^{-3/2})$, respectively.
It is also interesting to study optimality under the cost model for solving minimax problems~\cite{chen2025secondorder,chen2025solving} and nonlinear equations~\cite{liu2025enhanced}.
On the practical side, we hope our work can serve as a principled starting point for designing efficient (inexact) second-order optimizers in practice, especially given the ubiquity of lazy Hessian evaluation techniques in modern deep learning optimizers such as KFAC \citep{martens2015optimizing,grosse2016kronecker}, Sophia \citep{liu2024sophia}, Shampoo \citep{gupta2018shampoo,shi2023distributed,anil2020scalable}, SOAP \citep{vyas2025soap}, ASGO \citep{an2026asgo}, and Newton-Muon \citep{du2026newton}.




\bibliographystyle{plainnat}
\bibliography{sample}

@article{nesterov2018lectures,
  title={Lectures on Convex Optimization},
  author={Nesterov, Yurii},
  volume={87},
  year={2018},
  publisher={Springer},
  city={Cham}
}

@inproceedings{doikov2023second,
 title={Second-Order Optimization with Lazy {H}essians},
  author={Doikov, Nikita and Chayti, El Mahdi and Jaggi, Martin},
  booktitle={International Conference on Machine Learning},
  year={2023},
}

@article{chen2026optimal,
  title={Optimal Convex Optimization with Inexact Second-Order Oracles},
  author={Chen, Lesi and Liu, Chengchang and Luo, Luo and Lui, John and Zhang, Jingzhao},
  journal={arXiv preprint arXiv:2607.24520},
  year={2026}
}

@article{monteiro2013accelerated,
  title={An accelerated hybrid proximal extragradient method for convex optimization and its implications to second-order methods},
  author={Monteiro, Renato D.C. and Svaiter, B. F.},
  journal={SIAM J. Optim.},
  volume={23},
  number={2},
  pages={1092--1125},
  year={2013},
  publisher={SIAM}
}

@article{chayti2023unified,
  title={Unified Convergence Theory of Stochastic and Variance-Reduced Cubic Newton Methods},
  author={Chayti, El Mahdi and Doikov, Nikita and Jaggi, Martin},
  journal={Trans. Mach. Learn. Res.},
  year={2024}
}

@article{monteiro2012iteration,
  title={Iteration-complexity of a {N}ewton proximal extragradient method for monotone variational inequalities and inclusion problems},
  author={Monteiro, Renato D.C. and Svaiter, B. F.},
  journal={SIAM J. Optim.},
  volume={22},
  number={3},
  pages={914--935},
  year={2012},
  publisher={SIAM}
}

@inproceedings{adil2025balancing,
  title={Balancing Gradient and Hessian Queries in Non-Convex Optimization},
  author={Adil, Deeksha and Bullins, Brian and Sidford, Aaron and Zhang, Chenyi},
  booktitle={Advances in Neural Information Processing Systems},
  year={2025}
}

@inproceedings{chen2025solving,
  title={Solving Convex-Concave Problems with $\tilde{O}(\epsilon^{-4/7})$ Second-Order Oracle Complexity},
  author={Chen, Lesi and Liu, Chengchang and Luo, Luo and Zhang, Jingzhao},
  booktitle={Annual Conference on Learning Theory},
  pages={952--982},
  year={2025},
  organization={PMLR}
}

@article{nesterov2006cubic,
  title={Cubic regularization of Newton method and its global performance},
  author={Nesterov, Yurii and Polyak, Boris T.},
  journal={Math. Program.},
  volume={108},
  number={1},
  pages={177--205},
  year={2006},
  publisher={Springer}
}

@article{arjevani2019oracle,
  title={Oracle complexity of second-order methods for smooth convex optimization},
  author={Arjevani, Yossi and Shamir, Ohad and Shiff, Ron},
  journal={Math. Program.},
  volume={178},
  number={1},
  pages={327--360},
  year={2019},
  publisher={Springer}
}

@article{nesterov2008accelerating,
  title={Accelerating the cubic regularization of Newton’s method on convex problems},
  author={Nesterov, Yurii},
  journal={Math. Program.},
  volume={112},
  number={1},
  pages={159--181},
  year={2008},
  publisher={Springer}
}

@inproceedings{carmon2022optimal,
  title={Optimal and adaptive {M}onteiro-{S}vaiter acceleration},
  author={Carmon, Yair and Hausler, Danielle and Jambulapati, Arun and Jin, Yujia and Sidford, Aaron},
  booktitle={Advances in Neural Information Processing Systems},
  year={2022},
  pages = {20338--20350}
}

@inproceedings{nesterov1983method,
  title={A method for solving the convex programming problem with convergence rate ${O}(1/k^2)$},
  author={Nesterov, Yurii},
  booktitle={Dokl akad nauk Sssr},
  volume={269},
  pages={543},
  year={1983}
}

@inproceedings{kovalev2022first,
  title={The first optimal acceleration of high-order methods in smooth convex optimization},
  author={Kovalev, Dmitry and Gasnikov, Alexander},
  booktitle={Advances in Neural Information Processing Systems},
  year={2022},
pages={35339--35351},
}

@inproceedings{
chen2025secondorder,
title={Second-Order Min-Max Optimization with Lazy {H}essians},
author={Lesi Chen and Chengchang Liu and Jingzhao Zhang},
booktitle={International Conference on Learning Representations},
year={2025}
}

@article{nemirovskij1983problem,
  title={Problem complexity and method efficiency in optimization},
  author={Nemirovskij, Arkadij Semenovi{\v{c}} and Yudin, David Borisovich},
  year={1983},
  publisher={Wiley-Interscience},
  note = {(in Russian)}
}

@article{anil2020scalable,
  title={Scalable second order optimization for deep learning},
  author={Anil, Rohan and Gupta, Vineet and Koren, Tomer and Regan, Kevin and Singer, Yoram},
  journal={arXiv preprint arXiv:2002.09018},
  year={2020}
}

@inproceedings{gupta2018shampoo,
  title={Shampoo: Preconditioned stochastic tensor optimization},
  author={Gupta, Vineet and Koren, Tomer and Singer, Yoram},
  booktitle={International Conference on Machine Learning},
  year={2018}
}

@article{shi2023distributed,
  title={A distributed data-parallel pytorch implementation of the distributed shampoo optimizer for training neural networks at-scale},
  author={Shi, Hao-Jun Michael and Lee, Tsung-Hsien and Iwasaki, Shintaro and Gallego-Posada, Jose and Li, Zhijing and Rangadurai, Kaushik and Mudigere, Dheevatsa and Rabbat, Michael},
  journal={arXiv preprint arXiv:2309.06497},
  year={2023}
}

@InProceedings{chen2026faster,
  title = 	 {Faster Newton Methods for Convex and Nonconvex Optimization in Gradient Complexity},
  author =       {Chen, Lesi and Liu, Chengchang and Luo, Luo and Zhang, Jingzhao},
  booktitle = 	 { Conference on Learning Theory},
  year = 	 {2026},
}

@inproceedings{liu2024sophia,
  title={Sophia: A scalable stochastic second-order optimizer for language model pre-training},
  author={Liu, Hong and Li, Zhiyuan and Hall, David and Liang, Percy and Ma, Tengyu},
  booktitle={International conference on learning representations},
  year={2024}
}

@article{du2026newton,
  title={The {N}ewton-{M}uon optimizer},
  author={Du, Zhehang and Su, Weijie},
  journal={arXiv preprint arXiv:2604.01472},
  year={2026}
}

@inproceedings{vyas2025soap,
  title={{SOAP}: Improving and stabilizing shampoo using adam for language modeling},
  author={Vyas, Nikhil and Morwani, Depen and Zhao, Rosie and Shapira, Itai and Brandfonbrener, David and Janson, Lucas and Kakade, Sham},
  booktitle={International Conference on Learning Representations},
  year={2025}
}

@article{an2026asgo,
  title={{ASGO}: Adaptive structured gradient optimization},
  author={An, Kang and Liu, Yuxing and Pan, Rui and Ren, Yi and Ma, Shiqian and Goldfarb, Donald and Zhang, Tong},
  journal={Advances in Neural Information Processing Systems},
  year={2026}
}

@inproceedings{martens2015optimizing,
  title={Optimizing neural networks with kronecker-factored approximate curvature},
  author={Martens, James and Grosse, Roger},
  booktitle={International conference on machine learning},
  year={2015}
}

@inproceedings{grosse2016kronecker,
  title={A kronecker-factored approximate fisher matrix for convolution layers},
  author={Grosse, Roger and Martens, James},
  booktitle={International Conference on Machine Learning},
  year={2016}
}

@inproceedings{liu2025enhanced,
  title={An Enhanced Levenberg--Marquardt Method via Gram Reduction},
  author={Liu, Chengchang and Luo, Luo and Lui, John CS},
  booktitle={Proceedings of the AAAI conference on artificial intelligence},
  volume={39},
  number={18},
  pages={18772--18779},
  year={2025}
}

\appendix

\section{Missing Proofs of the Lower Bound}

\subsection{Proof of Lemma \ref{lem:block-zero-chain}}

\begin{proof}
For the function $f_{2T}: \sR^d \rightarrow \sR$ defined in Eq. (\ref{eq:block-zc}), we define 
$\vy \in \sR^{2T}$ be defined by the vector whose $i$th element is
\begin{align*}
    y_{[i]} = \langle \vx, \vv_i \rangle, \quad i = 1,\cdots, 2T.
\end{align*}
Next, we show that there exists an adversarial strategy of choosing the orthonormal vector set $\{\vv_t \}_{t=1}^{2T}$ such that Eq. (\ref{eq:ortho-vectors}) holds. The key observation is that the chain-like structure of the function $f_{2T}$ allows the adversary to simulate the query of any $t$-iteration deterministic algorithm \texttt{A} of the form (\ref{eq:alg-class}) with the truncated function (using the same notation in rotated coordinates and for its pullback to $\vx$) that only depends on $\vv_1,\cdots,\vv_t$:
\begin{align} \label{eq:hard-ins-truncated-y}
\begin{split}
    f_t(\vy)
    \triangleq& f_{2T}(\vy) \mid \{ y_{[i]}  = 0, ~i  = t+1,\cdots,2T\} \\
    =& \lambda \left( (a- y_{[1]})^2 + \sum_{i=1}^{m-1} (y_{[i]} - y_{[i+1]})^2  \right)^{3/2} \\
    & +\lambda \sum_{s=1}^{\lfloor t/m \rfloor-1} \left( \sum_{i=sm}^{(s+1)m-1} (y_{[i]} - y_{[i+1]})^2  \right)^{3/2} + \lambda \left( \sum_{i=m \lfloor t/m \rfloor}^{t-1} (y_{[i]} - y_{[i+1]})^2 + (y_{[t]}^2) \right)^{3/2}.
\end{split} 
\end{align}
Here the displayed expansion applies to $m\le t<2T$. For $1\le t<m$, it is instead
\[
f_t(\vy)=\lambda\left((a-y_{[1]})^2+\sum_{i=1}^{t-1}(y_{[i]}-y_{[i+1]})^2+y_{[t]}^2\right)^{3/2}.
\]
In either case, $f_t$ is independent of $y_{[t+1]},\ldots,y_{[2T]}$.
Specifically, the adversary can define the vectors $\vv_1,\cdots,\vv_{2T}$ as follows:
\begin{enumerate}
    \item At the initialization $\vx_0 = \vzero$, pick any unit vector $\vv_1 \in \sR^{d}$, and return
    the oracle of $f_1$ at $\vx_0$. Its gradient is $-3\lambda a^2\vv_1$; when $m=1$, also return its Hessian $6\lambda a\vv_1\vv_1^\top$.
    \item For any $t = 1,\cdots,T-1$, repeat the following process: The adversary
    simulates the current query~$\vx_t $  by executing the same update of algorithm $\texttt{A}_t$, and then picks a new unit $\vv_{t+1} \in \sR^{d}$ that is orthogonal to all previous basis $\vv_{1},\cdots,\vv_{t}$ and queries $\vx_1,\cdots,\vx_t$:
    \begin{align} \label{eq:select-vk}
            \langle \vv_{t+1}, \vv_i \rangle = 0 \quad {\rm and} \quad \langle \vv_{t+1}, \vx_i \rangle =0, \quad \forall i = 1,\cdots,t.
        \end{align}
        Such a vector $\vv_{t+1} \in \sR^d$ always exists when $d \ge 2t+1$. After that, the adversary returns its oracle using that of the truncated function~$f_{t+1}$ rather than the full function $f_{2T}$. 
    \item Finally, for the algorithm output $\vx_T$, the adversary picks an orthonormal vector set $\{\vv_{T+1},\cdots, \vv_{2T}\}$ that is orthogonal to the subspace ${\rm span}\{\vx_1,\cdots, \vx_T, \vv_1,\cdots,\vv_T\}$, which is possible when $d \ge 3T$.
\end{enumerate}
It remains to prove that, in the $t$-th iteration, the oracle of the truncated function $f_{t}$ returned by the above adversarial strategy is consistent with the oracle of the full function $f_{2T}$ at its query point $\vx_{t-1}$. In the following, we verify this property for both gradient and Hessian queries:
\paragraph{Consistency of Gradient queries.} For $t=1,\ldots,T$, the $t$th response is evaluated at $\vx_{t-1}$. Write $(y_{t-1})_{[i]}=\langle\vx_{t-1},\vv_i\rangle$. The adversary ensures that $(y_{t-1})_{[i]}=0$ for all $i=t,\ldots,2T$. The restriction defining $f_t$ preserves all partial derivatives in coordinates $1,\ldots,t$. For $i>t$, every edge difference contributing to the $i$th gradient coordinate is zero. Consequently,
\begin{align*}
    \frac{\partial f_t(\vy_{t-1})}{\partial y_{[i]}}=\frac{\partial f_{2T}(\vy_{t-1})}{\partial y_{[i]}},\quad i=1,\ldots,2T.
\end{align*}
Thus the gradient responses are consistent.
\paragraph{Consistency of Hessian queries.} For $t=m,2m,\ldots,m\lfloor T/m\rfloor$, the algorithm queries $\sO_t(\vx_{t-1})$. At these block boundaries, the truncated function is
\begin{align*}
f_t(\vy)=&\lambda\left((a-y_{[1]})^2+\sum_{i=1}^{m-1}(y_{[i]}-y_{[i+1]})^2\right)^{3/2}+\lambda\sum_{s=1}^{t/m-1}\left(\sum_{i=sm}^{(s+1)m-1}(y_{[i]}-y_{[i+1]})^2\right)^{3/2}+\lambda|y_{[t]}|^3.
\end{align*}
The displayed complete blocks coincide with those of $f_{2T}$. All remaining blocks of $f_{2T}$ have zero edge differences at $\vy_{t-1}$, and $y_{[t]}=0$ there. Since the gradient and Hessian of $\|\cdot\|^3$ vanish at zero, the terminal term of $f_t$ and these remaining blocks have zero gradients and Hessians. Therefore,
\begin{align*}
\frac{\partial^2 f_t(\vy_{t-1})}{\partial y_{[i]}\partial y_{[j]}}=\frac{\partial^2 f_{2T}(\vy_{t-1})}{\partial y_{[i]}\partial y_{[j]}},\quad i,j=1,\ldots,2T.
\end{align*}
Hence the full Hessian responses are consistent at every Hessian query point $\vx_{t-1}$.
\end{proof}

\subsection{Proof of Lemma \ref{lem:subgap}}

\begin{proof}
Following the notations from the previous section, Eq. (\ref{eq:ortho-vectors}) implies that the last $T$ coordinates of the rotated variable $\vy_T$ are all zero. Therefore, with a change of variable from $\vx$ to $\vy$, we have $f_{2T}(\vy_T) = f_{T}(\vy_T)$, where $f_T$ is the truncated function defined in Eq. (\ref{eq:hard-ins-truncated-y}) with $t=T$. Next, let us lower-bound the value of $f_T(\vy)$ for any $\vy \in \sR^{2T}$. Using  Jensen's inequality with $\sum_{i=1}^m a_i^2 \ge \left( \sum_{i=1}^m a_i \right)^2 / m$ in each block of Eq. (\ref{eq:hard-ins-truncated-y}) yields that
\begin{align*}
    f_T(\vy) \ge& \frac{\lambda}{m^{3/2}} \left \vert a- y_{[m]}\right \vert^{3} + \frac{\lambda}{m^{3/2}} \sum_{s=1}^{\lfloor T/m \rfloor-1} \left\vert y_{[sm]} - y_{[(s+1)m]}  \right \vert^{3} + \frac{\lambda}{m^{3/2}} \left \vert  y_{[m \lfloor T/m \rfloor]} \right\vert^{3}, \quad \forall \vy \in \sR^{2T}.
\end{align*}
Next, we apply Jensen's inequality with $\sum_{i=1}^n |a_i|^3 \ge \left( \sum_{i=1}^n |a_i| \right)^3 / n^2$ for $n = \lfloor T/m \rfloor +1$ terms in the above to obtain that
\begin{align*}
    f_T(\vy) \ge \frac{\lambda}{m^{3/2} \left( \lfloor T/m \rfloor +1  \right)^2} \left( \left \vert a- y_{[m]}\right \vert + \sum_{s=1}^{\lfloor T/m \rfloor-1} \left\vert y_{[sm]} - y_{[(s+1)m]}  \right \vert + \left \vert  y_{[m \lfloor T/m \rfloor]} \right\vert \right)^3, \quad \forall \vy \in \sR^{2T}.
\end{align*}
To proceed, we apply the triangle inequality with $\sum_{i=1}^n \vert a_i \vert \ge \left \vert \sum_{i=1}^n  a_i  \right \vert$ to further obtain that 
\begin{align*}
    f_T(\vy) \ge \frac{\lambda a^3}{m^{3/2} \left( \lfloor T/m \rfloor +1  \right)^2}, \quad \forall \vy \in \sR^{2T}.
\end{align*}
\end{proof}

\subsection{Proof of Theorem \ref{thm:lower-bound-convex}}

\begin{proof}
    From the form of the function $f_{2T}(\vx)$ defined in Eq. (\ref{eq:block-zc}), it is easy to see that its minimizer $\vx^*$ is achieved under the following constraint:
\begin{align*}
    \langle \vx^*, \vv_i \rangle = a, \quad \forall i = 1,\cdots,2T.
\end{align*}
Taking $\vx^*=a\sum_{i=1}^{2T}\vv_i$, we have $\Vert \vx^* \Vert = a \sqrt{2T}$, which means that $\Vert \vx^* \Vert \le D$ if we set $a = D / \sqrt{2T}$. Next, we justify the choice of~$\lambda$ that can ensure the function to have $L$-Lipschitz continuous Hessians. Note that $f_{2T}(\vx)$ is the summation of $\lfloor 2T/m \rfloor+1$ blocks, and each coordinate only appears at most in two adjacent blocks. Therefore, it suffices to upper-bound the Hessian Lipschitz constant of each block, and the final constant is at most multiplied by~2. Define the matrix $\mA \in \sR^{m \times (m+1)}$ as
\begin{align*}
\mA = 
\begin{bmatrix}
        1 & -1 & 0 & \cdots & 0 \\
        0 & 1 & -1 & \cdots & 0 \\
        \vdots &  \ddots & \ddots & \vdots \\
        0 & \cdots & \cdots & 1 & -1 \\
\end{bmatrix}.
\end{align*}
Then each block has the form of 
\begin{align*}
    g_s(\vy) = \lambda \Vert \mA \vy^{(s)} \Vert^3, \quad \vy^{(s)} = \left( \langle \vx, \vv_{sm} \rangle, \cdots, \langle \vx, \vv_{(s+1)m} \rangle \right)^\top.
\end{align*}
Since the function $\Vert \,\cdot\, \Vert^3$ has $6$-Lipschitz continuous Hessians \citep[Lemma 4.2.4]{nesterov2018lectures} and $\Vert \mA \Vert \le 2$, it is easy to check that $g_s(\vy)$ has $48 \lambda$-Lipschitz continuous Hessians. Hence, we can conclude that the final function $f_{2T}(\vx)$ has $96\lambda$-Lipschitz continuous Hessians. In other words, we can let $\lambda = L / 96$ to let the Hessian Lipschitz constant be bounded by $L$. The first and final (possibly incomplete) blocks obey the same bound by restricting the difference matrix and allowing an affine shift. Finally, since $T\ge m$ gives $\lfloor T/m\rfloor+1\le2T/m$ and $f_{2T}(\vx^*) = 0$, substituting the parameters $a = D / \sqrt{2T}$ and $\lambda = L / 96$ into inequality (\ref{eq:lb-suboptimality-gap})  yields that
\begin{align*}
    f_{2T}(\vx_T) - f_{2T}(\vx^*) = \Omega \left( \frac{\sqrt{m} L D^3}{T^{7/2}} \right).
\end{align*}
 In other words, to find an $\epsilon$-solution the algorithm has to take at least $T = \Omega( m^{1/7} (L D^3 /\epsilon)^{2/7} )$ iterations.
\end{proof}

\section{Missing Proofs of the Upper Bound}

\subsection{Proof of Lemma \ref{lem:reduce-gap}}

\begin{proof}
Lemma~\ref{lem:sub-prob-grad-dom}(2) shows that
\begin{equation}\label{eq:reduce-gap-certificate}
f_{\vx,\gamma}(\vy_0)-f_{\vx,\gamma}^*\le\frac{2\sqrt2}{3\sqrt\gamma}
\|\nabla f_{\vx,\gamma}(\vy_0)\|^{3/2}=\Delta_0.
\end{equation}
For convenience, write $r_*=\|\vy_{\vx,\gamma}^*-\vx\|$. Then, for the initialization $\vz_0=\mathcal{C}_H(\vx)$ and
$\vy_0=\mathcal{C}_H(\vz_0)$, applying Lemma~\ref{lem:crn}(2) to the
first step and Lemma~\ref{lem:crn}(1) to the second gives
\[
f_{\vx,\gamma}(\vz_0)-f_{\vx,\gamma}^*\le\frac H3 r_* ^3,
\qquad
\|\nabla f_{\vx,\gamma}(\vy_0)\|^{3/2}
\le3\sqrt H\bigl(f_{\vx,\gamma}(\vz_0)-f_{\vx,\gamma}^*\bigr)
\le H^{3/2} r_*^3.
\]
Therefore $\Delta_0\le(2\sqrt2/3)\gamma\kappa^{3/2}r_*^3
\le50\gamma\kappa^{3/2}r_*^3$. Condition~\ref{cond:gapreduce}
and the stated choice of $S$ imply
\begin{equation}\label{eq:reduce-gap-final-certificate}
f_{\vx,\gamma}(\vy_S)-f_{\vx,\gamma}^*\le\Delta_S=(3/4)^S\Delta_0
\le\frac{\gamma r_*^3}{384^{3/2}\sqrt\kappa}.
\end{equation}
Applying Lemma~\ref{lem:crn}(1) to the final step
$\hat\vy=\mathcal{C}_H(\vy_S)$ yields
\[
\|\nabla f_{\vx,\gamma}(\hat\vy)\|
\le3^{2/3}H^{1/3}\Delta_S^{2/3}
\le\frac{\gamma r_*^2}{128}.
\]
The second inequality in (\ref{eq:cubic-grad-uniform}) then gives
$\|\hat\vy-\vy_{\vx,\gamma}^*\|\le2^{1/3}r_*/8\le r_*/4$.
Thus $\|\hat\vy-\vx\|\ge3r_*/4$, and consequently the output $\hat \vy$ provably satisfies Condition \ref{cond:cubic-proxi}. 
\[
\|\nabla f_{\vx,\gamma}(\hat\vy)\|
\le\frac{\gamma}{72}\|\hat\vy-\vx\|^2
\le\frac{\gamma}{32}\|\hat\vy-\vx\|^2.
\]
\end{proof}

\subsection{Proof of Lemma \ref{lem:finite-segment-search}}

\begin{proof}

We verify that $M/\Delta=\mathcal{O}(\kappa)$. Let $R=(6\Delta/\gamma)^{1/3}$. The second inequality in~(\ref{eq:cubic-grad-uniform}) gives, \[
f_{\vx,\gamma}(\vw)-f_{\vx,\gamma}^*\ge\frac{\gamma}{6} \|\vw-\vy_{\vx,\gamma}^*\|^3, \qquad \forall \vw \in \sR^d.
\] 
In particular, $\|\vy-\vy_{\vx,\gamma}^*\|\le R$. Let $\Xi=(J+1)^{4/3}\kappa^{2/3}$ and $W = \sum_{i=0}^J \lambda^i$. Then, $A_k \le k^2 \le N^2 = \gO(\Xi)$. Since $g(\vx_k)\ge W f_{\vx,\gamma}^*$, Lemma~\ref{lem:comparator} and inequality~(\ref{eq:upper-potential-bound}) imply
\begin{align} \label{eq:potential-bound}
\|\vv_k-\bar{\vv}\|_\mB^2/2
\le E_0+A_k(0.643\Delta+\zeta)=\mathcal{O}(\Xi \Delta).    
\end{align}
By the definition of $\mB$ in Eq. (\ref{eq:upper-metric}), we have \(\mB \succeq 100 (J+1)^{4/3}(H^2 \Delta)^{1/3} \mI_d\). Therefore, we can obtain that
\[
\| \vv_k - \bar \vv \|^2 = \gO( \kappa^{2/3} \Delta^{2/3} / H^{2/3} ) = \gO( (\Delta/\gamma)^{2/3}) = \gO(R^2).
\]
This implies $\| \vv_k - \bar \vv \| =\gO(R) $. Also, by the definition of $\bar \vv$ in Lemma \ref{lem:comparator}, we have \( \| \bar \vv - \vy \| = c^{-J} \| \vy - \vy_{\vx,\gamma}^* \| \le R/2 \) for $c = 1+1/J$, so the triangle inequality immediately gives $\| \vv_k - \vy_{\vx,\gamma}^* \| = \gO(R) $. Furthermore, we recall inequality (\ref{eq:upper-individual-distance}) that $\|\vx_k-\vy_{\vx,\gamma}^*\|\le W^{1/3}R \le 2R$ and $\|\vx_k-\vy\|\le(1+W^{1/3})R\le3R$. Since every point of $\vy' \in [\vy_{\vx,\gamma}^*,\vx_k]$ is within $3R$ of $\vy$, Hessian Lipschitzness by Lemma \ref{lem:sub-prob-grad-dom} gives 
\[
\nabla^2f_{\vx,\gamma}(\vy')\succeq \nabla^2f_{\vx,\gamma}(\vy)-3HR\mI_d, \qquad \forall \vy' \in [\vy_{\vx,\gamma}^*,\vx_k].
\]
Since $\nabla f_{\vx,\gamma}(\vy_{\vx,\gamma}^*) =0$ at the minimizer $\vy_{\vx,\gamma}^*$, integrating from $\vy_{\vx,\gamma}^*$ yields
\begin{align*}
    f_{\boldsymbol{x}, \gamma}(\boldsymbol{x}_k) - f_{\boldsymbol{x}, \gamma}^* =&  \int_0^1 (1-t) (\boldsymbol{x}_k - \vy_{\vx,\gamma}^*)^\top \nabla^2 f_{\boldsymbol{x}, \gamma}(\vy_{\vx,\gamma}^* + t(\vx_k - \vy_{\vx,\gamma}^*)) (\boldsymbol{x}_k - \vy_{\vx,\gamma}^*) \, {\rm d} t \\
    \ge& \left( \int_0^1 (1-t) \, {\rm d} t \right) \cdot \left( \Vert{}\boldsymbol{x}_k - \vy_{\vx,\gamma}^*\Vert{}_{\nabla^2f_{\vx,\gamma}(\vy)}^2 - 3HR \Vert{}\boldsymbol{x}_k - \vy_{\vx,\gamma}^*\Vert{}^2 \right).
\end{align*}
Using $\int_0^1 (1-t) \, {\rm d} t =1/2$ and rearranging the above inequality gives
\[
\|\vx_k-\vy_{\vx,\gamma}^*\|_{\nabla^2f_{\vx,\gamma}(\vy)}^2
\le2\bigl(f_{\vx,\gamma}(\vx_k)-f_{\vx,\gamma}^*\bigr)
+3HR\|\vx_k-\vy_{\vx,\gamma}^*\|^2
=\mathcal{O}(\kappa\Delta).
\]
The definition of $\mB$ in Eq. (\ref{eq:upper-metric}) also indicates that
$\mB\succeq2(J+1) \nabla^2 f_{\vx,\gamma}(\vy)$. Hence, the potential bound~(\ref{eq:potential-bound}) indicates that \[
\| \vv_k - \bar \vv \|_{\nabla^2 f_{\vx,\gamma}(\vy)}^2  \le \frac{1}{2(J+1)}  \| \vv_k - \bar \vv \|_{\mB}^2  = \gO( (J+1)^{1/3} \kappa^{2/3} \Delta ).
\] 
In addition, Lemma~\ref{lem:upper-snapshot-distance} and $\bar{\vv}-\vy_{\vx,\gamma}^*=(1-c^{-J})(\vy-\vy_{\vx,\gamma}^*)$ give 
\[
\| \bar \vv - \vy_{\vx,\gamma}^*\|_{\nabla^2 f_{\vx,\gamma}(\vy)}^2 =  \gO(\kappa^{2/3} \Delta).
\]
Thus, by the triangle inequality, it is easy to obtain that
\[
\|\vv_k-\vy_{\vx,\gamma}^*\|_{\nabla^2 f_{\vx,\gamma}(\vy)}^2 = 
\mathcal{O}\bigl(((J+1)^{1/3}\kappa^{2/3}+\kappa^{2/3})\Delta\bigr)
=\mathcal{O}(\kappa\Delta),
\]
where $J=\mathcal{O}(\log\kappa)=\mathcal{O}(\kappa)$. Combining all the above analyses, we have 
\[
\| \vv_k - \vx_k  \|_{\nabla^2 f_{\vx,\gamma}(\vy)}^2 = \gO(\kappa \Delta), \quad \| \vv_k - \vx_k  \| = \gO(R), \quad H \rho \| \vv_k - \vx_k \|^2 = \gO(HR^3) = \gO(\kappa \Delta).
\]
Recalling the definition of $M$ in Eq. (\ref{eq:upper-line-lipschitz}), we have $M = \gO(\kappa \Delta)$. 
Since $\zeta=\Delta/100$, each binary search uses at most $3+\lceil\log_2^+(M/\zeta)\rceil=\mathcal{O}(\log\kappa)$ gradients of $g$, including the two endpoint queries. Each gradient of $g$ costs $J+1$ gradients of $f_{\vx,\gamma}$. Adding the $N$ descent-gradient queries to the cost of the $N$ searches, and using $N=\mathcal{O}((J+1)^{2/3}\kappa^{1/3})$ and $J+1=\mathcal{O}(\log\kappa)$, gives a total of $\mathcal{O}(N(J+1)(1+\log\kappa))=\mathcal{O}(\kappa^{1/3}\log^{8/3}\kappa)$ gradient evaluations of $f_{\vx,\gamma}$, as claimed in Lemma~\ref{lem:upper-gap-reduction}.


\end{proof}

\end{document}